\documentclass[11pt,reqno]{amsart}
\usepackage[a4paper, margin=1.2in]{geometry}
\usepackage{amsmath, amsthm, amssymb}
\usepackage{graphicx, array}
\usepackage{longtable}
\usepackage{epstopdf}
\usepackage{xcolor}
\usepackage{tikz-cd}
\usepackage{cite}
\usetikzlibrary{arrows.meta}
\usepackage{wrapfig}
\usetikzlibrary{calc}
\usetikzlibrary{backgrounds}
\usetikzlibrary{automata}
\usetikzlibrary{positioning}
\usepackage{scalerel}

\makeatletter %this allows item labels to be referenced, note it doesn't work with hyperref
\def\namedlabel#1#2{\begingroup
    #2%
    \def\@currentlabel{#2}%
    \label{#1}\endgroup
}
\makeatother

\newtheorem{Thm}{Theorem}[section]
\newtheorem{Prop}[Thm]{Proposition}
\newtheorem{Lem}[Thm]{Lemma}
\newtheorem{Cor}[Thm]{Corollary}

\newtheorem*{Thm*}{Theorem}

{\theoremstyle{definition}
}
{\theoremstyle{definition}
\newtheorem{Rem}[Thm]{Remark}}
{\theoremstyle{definition}
\newtheorem{Ex}[Thm]{Example}}

\newenvironment{Proof}{\rm \trivlist\item[\hskip \labelsep{\bf
Proof.\quad}]}{\hfill\qed\par\medskip\endtrivlist}

\newcommand{\G}{\mathcal G}

\newcommand{\supp}{\operatorname{supp}}

\renewcommand{\int}{\operatorname{int}}
\renewcommand{\S}{S^\sharp} %nonzero elements
\newcommand{\tight}[1]{K\G_T({#1})} %tight Steinberg algebra corresponding to S
\renewcommand{\empty}{\varepsilon} %empty word
\newcommand{\sing}[2]{I_{#1}(#2)} % singular ideal

\newcommand{\tightid}[2]{{\mathcal T}_{#1}(#2)} % tight ideal

\newcommand{\ol}{\overline}

\newcommand{\eq}[1]{\equiv_{#1}}
\newcommand{\N}[3]{\mathcal N_{#1}(#2,#3)}
\newcommand{\Sg}{\mathcal{S}\Gamma} %the underlying graph of the simplicity automaton
\newcommand{\Vmin}{V_{\min}} %the vertices in the minimal components
\newcommand{\Vess}{V_{\mathrm{ess}}} %essential vertices

\numberwithin{equation}{section}

\title[On the simplicity of Nekrashevych algebras]{On the simplicity of Nekrashevych algebras of contracting self-similar groups}

\author{Benjamin Steinberg}
\address[B.~Steinberg]{%
    Department of Mathematics\\
    City College of New York\\
    Convent Avenue at 138th Street\\
    New York, New York 10031\\
    USA}
\email{bsteinberg@ccny.cuny.edu}

\author{N\'ora Szak\'acs}
\address[N.~Szak\'acs]{%
Department of Mathematics\\ University of York \\
Heslington, York, YO10 5DD \\ UK\\}
\email{szakacsn@math.u-szeged.hu}

\thanks{\textcolor{white}{.}}
\thanks{\noindent
\setlength\intextsep{0pt}
\begin{wrapfigure}{l}{0cm}
\includegraphics[height=3em]{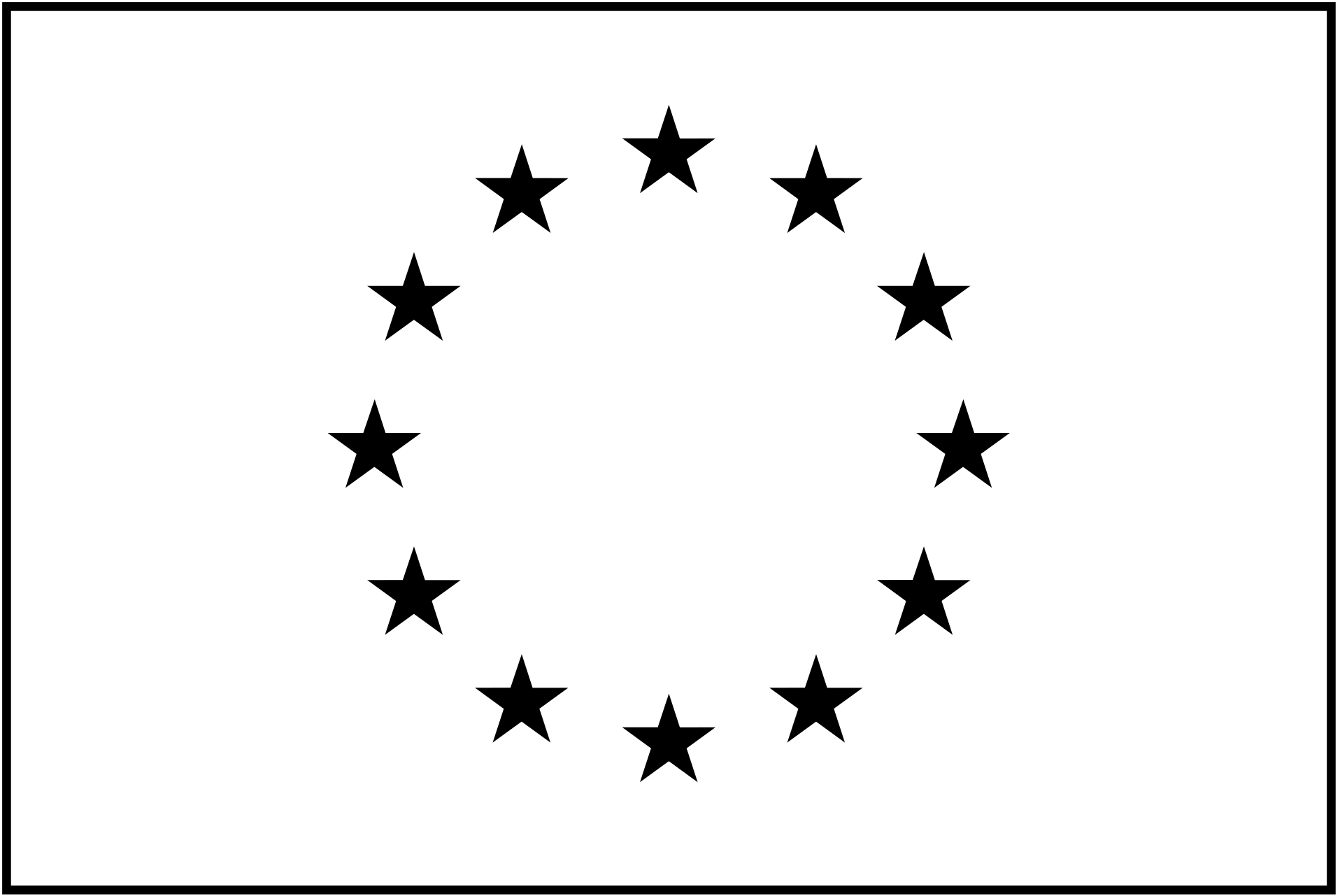}
\end{wrapfigure}
The second author was funded by the European Union’s
Horizon 2020 research and innovation programme under the
Marie Sk{\l}odowska-Curie grant agreement No 799419.\\
}

\date{\today}

\begin{document}

\begin{abstract}
Nekrashevych algebras of self-similar group actions are natural generalizations of the classical Leavitt algebras.  They are discrete analogues of the corresponding Nekrashevych $C^\ast$-algebras.  In particular, Nekrashevych, Clark, Exel, Pardo, Sims and Starling have studied the question of simplicity of Nekrashevych algebras, in part, because non-simplicity of the complex algebra implies non-simplicity of the $C^\ast$-algebra.

In this paper we give necessary and sufficient conditions for the Nekrashevych algebra of a contracting group to be simple. Nekrashevych algebras of contracting groups are finitely presented.  We give an algorithm which on input the nucleus of the contracting group, outputs all characteristics of fields over which the corresponding Nekrashevych algebra is simple.  Using our methods, we determine the fields over which the Nekrashevych algebras of a number of well-known contracting groups are simple including the Basilica group, Gupta-Sidki groups, \textsf{GGS}-groups, multi-edge spinal groups,  \v{S}uni\'{c} groups associated to polynomials (this latter family includes the Grigorchuk group, Grigorchuk-Erschler group and Fabrykowski-Gupta group) and self-replicating spinal automaton groups.
\end{abstract}

\maketitle

\section{Introduction}
The theory of self-similar groups began in the seventies with the Ukrainian school, who studied groups generated by finite automata in hopes of finding finitely generated, infinite torsion groups; see~\cite{GNS}.  The modern theory of self-similar groups was laid out by Nekrashevych in his seminal monograph~\cite{selfsimilar}.  A self-similar group is a group $G$ acting on a regular rooted $n$-ary tree in such a way that the self-similarity of the tree is reflected in the group.  Early famous examples of self-similar groups include the celebrated Grigorchuk $2$-group~\cite{grigroup}, the first known group of intermediate growth (solving a well-known problem of Milnor), and the Gupta-Sidki $p$-groups~\cite{GuptaSidki}; these are among the easiest to understand examples of finitely generated,  infinite torsion groups.  The Basilica group was the first example of an amenable group that is not constructible from groups of subexponential growth using the usual operations preserving amenability~\cite{GrigZuk,BV05}.  The Basilica group belongs to  the class of iterated monodromy groups, which is an important subclass of self-similar groups.  Iterated monodromy groups were introduced by Nekrashevych to study dynamical systems~\cite{selfsimilar} and played a key role in the solution of Hubbard's twisted rabbit problem in complex dynamics~\cite{rabbit}.  From the dynamical systems point of view, contracting self-similar groups are perhaps the most important class of self-similar groups.  These groups have a certain limit space upon which they act, and, in the case of iterated monodromy groups coming from complex dynamics, they encode the dynamics on the Julia set~\cite{selfsimilar}.

In an influential paper~\cite{Nekcstar}, Nekrashevych associated a $C^\ast$-algebra to each self-similar group.  In the case that the group is trivial, the $C^\ast$-algebra is the Cuntz algebra~\cite{cuntz}, which is a famous finitely presented simple $C^\ast$-algebra.  The Nekrashevych $C^\ast$-algebra is the algebra of a minimal and effective ample groupoid. The Steinberg algebra~\cite{mygroupoidalgebra} of this groupoid is the algebraic counterpart of the Nekrashevych $C^\ast$-algebra and is called the \emph{Nekrashevych algebra} of the self-similar group (with coefficients in a field).   Nekrashevych algebras over fields were first studied in~\cite{Nekrashevychgpd};  further work was done in~\cite{algebraicExelPardo, gridealSteinberg} and by the authors in~\cite{simplicity}.  When the group is trivial, the Nekrashevych algebra is the Leavitt algebra~\cite{Leavitt}, a celebrated finitely presented simple algebra. The Nekrashevych algebra of a self-similar group can be given by generators and relations, and  in the case of a contracting self-similar group the algebra is finitely presented~\cite{Nekcstar}.

The groupoid  associated to a self-similar group action  is amenable if the group is amenable~\cite{ExPadKatsura}.  The groupoid is also amenable for self-similar groups that are self-replicating and contracting~\cite[Theorem~5.6]{Nekcstar}.  Minimal and effective Hausdorff ample groupoids have simple Steinberg algebras and, if amenable, have simple $C^\ast$-algebras; cf.~\cite{operatorsimple1,operatorguys2,groupoidprimitive,Renault}.  Unfortunately, the groupoids associated to self-similar groups are rarely Hausdorff.  Nonetheless, the recent results of~\cite{nonhausdorffsimple} provide an avenue to studying simplicity of these $C^\ast$-algebras.  In particular, the results of~\cite{nonhausdorffsimple} imply that if the complex Steinberg algebra of the groupoid is not simple, then the $C^\ast$-algebra is not simple and so it is natural to look at Nekrashevych algebras over fields and study when they are simple.  This was first done by Nekrashevych, himself, in~\cite{Nekrashevychgpd}.

Nekrashevych showed in~\cite{Nekrashevychgpd} that the Nekrashevych algebra of the Grigorchuk group is not simple in characteristic $2$.  Clark \textit{et al.}~\cite{nonhausdorffsimple} showed that the Nekrashevych algebra of the Grigorchuk group is simple over all fields of characteristic different than $2$ and that the Nekrashevych $C^\ast$-algebra of the Grigorchuk group is simple.  In an unpublished note, Nekrashevych showed that the Nekrashevych algebra of the Grigorchuk-Erschler group~\cite{Grigdegree,Erschler} is not simple over any field.  This is a contracting, self-replicating and amenable group and its corresponding $C^\ast$-algebra is the first example of an algebra of a minimal, effective and amenable ample (non-Hausdorff) groupoid that is not simple of which we are aware.

In this paper we give a complete characterization of when a contracting self-similar group $G$ has a simple Nekrashevych algebra in terms of the nucleus of the group.  It turns out that either the Nekrashevych algebra is simple over no field, or it is simple over fields of all but finitely many positive characteristics.  We present an algorithm that on input the nucleus, which is a certain finite automaton, outputs the characteristics over which the algebra is simple.  Also, we provide for each finite set $\mathcal P$ of primes, a contracting self-similar group whose Nekrashevych algebra is simple over precisely those fields whose characteristic does not belong to $\mathcal P$. We also show that the Nekrashevych algebra of a self-similar group $G$ is simple if and only if its natural representation on the $K$-vector space with basis the boundary of the tree is faithful.

Using the algorithm and its theoretical underpinnings, we analyze the simplicity of the Nekrashevych algebras of a number of well-known contracting self-similar groups.   In particular, we show that the Basilica group~\cite{BasilicaGZ}\footnote{It was pointed out to us by Nekrashevych that the Basilica group has a Hausdorff associated groupoid, so our methods are not needed for this case.}, the Gupta-Sidki $p$-groups~\cite{GuptaSidki} and the Fabrykowski-Gupta group~\cite{FabGupta} have simple Nekrashevych algebras over every field.   We show that \textsf{GGS}-groups have simple Nekrashevych algebras over every field when defined over an alphabet of prime power size, but over alphabets of size divisible by more than one prime they can have algebras that are simple over every field or over no field, and we characterize which situation occurs. We show that for multi-edge spinal groups~\cite{multiedgespinal}, the Nekrashevych algebra is never simple unless the group is a \textsf{GGS}-group.  Moreover, our methods give fairly short proofs of simplicity/non-simplicity, when compared to the long topological arguments involving ample groupoids found, for example, in~\cite{nonhausdorffsimple}.

 We give a complete characterization of simplicity of Nekrashevych algebras of \v{S}uni\'{c} groups associated to polynomials over finite prime fields~\cite{sunicgroups}.  These include the Grigorchuk group and the Grigorchuk-Erschler group.  In particular, for each prime $p$ we show that the  Nekrashevych algebra of the  \v{S}uni\'{c} group associated to a primitive polynomial over $\mathbb Z_p$ (which is a finitely generated, infinite $p$-group of intermediate growth) is simple over every field except those of characteristic $p$; the Grigorchuk group is the case $p=2$ and the primitive polynomial is $x^2+x+1$, which is the smallest possible example.

Many of these examples belong to a class of self-similar contracting groups that we call multispinal groups, which simultaneously generalizes \textsf{GGS}-groups, multi-edge spinal groups, \v{S}uni\'{c} groups and self-replicating spinal automaton groups~\cite{spinalgroups}.  For these groups, the criterion for simplicity of the Nekrashevych algebra boils down to the representation theory of finite groups, i.e., to Fourier analysis on finite groups.  After the first verison of our paper was placed on ArXiv, Yoshida proved that the Nekrashevych algebra of a multispinal group over the complex numbers is simple if and only if  its Nekrashevych $C^\ast$-algebra is simple~\cite{multispinalCstar}.

Our approach to studying simplicity is based on inverse semigroup algebras.  Associated to any self-similar group $G$ is a congruence-free inverse semigroup $S$.
The ample groupoid associated to the self similar group $G$ is the tight groupoid of $S$ (in the sense of Exel~\cite{Exel}).
In~\cite{simplicity}, the authors show that the Steinberg algebra of the tight groupoid of a congruence-free inverse semigroup has a unique maximal ideal and give an explicit description of its preimage in the algebra of the inverse semigroup, called the singular ideal.  We use this description of the singular ideal, specialized to the case of inverse semigroups associated to self-similar group actions, to attack the simplicity question for Nekrashevych algebras.  In particular, we show that if the singular ideal contains any element representing a non-trivial element of the Nekrashevych algebra, then there is such an element coming from the group algebra of $G$.  In the case of a contracting group, we show such an element can be assumed to be supported on the nucleus.  We then encode the question of whether the nucleus supports a singular element into the combinatorics of a certain finite labeled graph.

The paper is organized as follows.  Section~\ref{s:inverse.sgps} is a preliminary section on inverse semigroups and their associated algebras.  This is followed by a section recalling fundamental definitions and notions from the theory of self-similar groups.  Section~\ref{s:na} recalls the definition of the Nekrashevych algebra and proves in the algebraic setting some analogues of results from the $C^\ast$-algebraic setting~\cite{Nekcstar}.  After some preliminary results on the tight ideal and the singular ideal of the semigroup algebra of the inverse semigroup associated to a self-similar group action, we prove our main result on simplicity of Nekrashevych algebras of contracting groups in Section~\ref{s:simple}; some useful intermediary results that apply to Nekrashevych algebras of arbitrary self-similar groups are also included. Section~\ref{s:ex} applies the results of the previous section to the Basilica group; we also provide a simple method to construct examples of non-simple Nekrashevych algebras.  In Section~\ref{s:multispinal}, we introduce multispinal groups, which include many of the most prominent families of contracting self-similar groups.  A criterion for simplicity of the Nekrashevych algebra of a multispinal group is given in terms of the representation theory of finite groups.  The criterion is applied to Gupta-Sidki groups, \textsf{GGS}-groups, multi-edge spinal groups and \v{S}uni\'{c} groups (generalizing the Grigorchuk  and Grigorchuk-Erschler groups).

\section{Inverse semigroups and their algebras}\label{s:inverse.sgps}
An \emph{inverse semigroup} is a semigroup $S$ such that, for each $s\in S$, there is a unique element $s^\ast\in S$ with $ss^\ast s=s$ and $s^\ast ss^\ast=s^\ast$. Note that $(st)^\ast =t^\ast s^\ast$ and $(s^\ast)^\ast=s$. The set $E(S)$ of idempotents of $S$ is a commutative subsemigroup and  is a meet semilattice with respect to the partial ordering $e\leq f$ if $ef=e$.  Moreover, the meet is given by the product.  A key example of an inverse semigroup is the symmetric inverse monoid on a set $X$, that is, the monoid of all partially defined injective mappings from $X$ to $X$ under composition of partial functions.  Every inverse semigroup can be faithfully represented as an inverse semigroup of partial injective mappings.  See~\cite{Lawson} for an introduction to inverse semigroups.

An inverse semigroup $S$ with zero is \emph{congruence-free} if it admits no proper non-zero quotients.  An inverse semigroup with zero is well known to be congruence-free if and only if it satisfies the following three conditions~\cite{Baird}: $SsS=S$ for all $s\neq 0$; $E(S)$ is a maximal commutative subsemigroup; and if $0\neq f<e\in E(S)$, then $fg=0$ for some $0\neq g<e$.  In the next section, we shall be interested in a congruence-free inverse semigroup associated to a self-similar group action.

If $K$ is a field and $S$ is an inverse semigroup with zero, the \emph{contracted semigroup algebra} $K_0S$ of $S$  is the $K$-algebra with basis $\S=S\setminus\{0\}$ and multiplication extending that of $S$ where we identify the zeroes of $S$ and $K$.  For $a=\sum_{s\in \S}a_ss$ in $K_0S$, we denote by $\supp a$ the set of $s\in \S$ with $a_s\neq 0$.

In this paper, we will be interested in a certain quotient of the contracted inverse semigroup algebra, first introduced by Exel in the $C^\ast$-algebraic setting~\cite{Exel}.
Let $E$ be a semilattice. We say that $F \subseteq E$ \emph{covers} $e \in E$ if $f \leq e$ for all $f \in F$, and if $0\neq f'\leq e$, then $f'f\neq 0$ for some $f\in F$.

Given any inverse semigroup $S$, the \emph{tight ideal} $\tightid{K}{S}$ of $K_0S$ is the ideal generated by all products $\prod_{f \in F}(e-f)$, where $e \in E(S)$, and  $F \subseteq E(S)$ is a finite cover of $e$. This ideal arises naturally as the kernel of the surjective homomorphism $K_0S \to \tight{S}$, where $\tight{S}$ is the Steinberg algebra~\cite{mygroupoidalgebra} of the tight groupoid $\G_T(S)$ of the inverse semigroup $S$ as defined by Exel~\cite{Exel}; see~\cite{groupoidprimitive,simplicity} for details.

The singular ideal of $K_0S$ was introduced by the authors in~\cite{simplicity}.  An element $a\in K_0S$ is \emph{singular} if, for all $0\neq e\in E(S)$, there exists $0\neq f\leq e$ with $af=0$.  The singular elements form a two-sided ideal $\sing K S$ containing $\tightid K S$ called the \emph{singular ideal}.  One of the main results of~\cite{simplicity} is then the following theorem.

\begin{Thm}
\label{t:cong-free.simple}
Let $S$ be a congruence-free inverse semigroup and $K$ a field.  Then $\sing{K}{S}$ is the unique maximal ideal of $K_0S$ containing $\tightid K S$.  In particular, $K_0S/\tightid{K}{S}$ is simple if and only if $\sing K S = \tightid K S$.
\end{Thm}

\section{Self-similar groups and associated monoids}\label{s:ssg}
Let $X$ be a finite alphabet with $|X|\geq 2$.  The free monoid on $X$ is denoted $X^\ast$.  A \emph{self-similar group} over $X$ is a group $G$ with a faithful action on $X^\ast$ by length-preserving permutations such that, for each $x\in X$, there is $g|_x\in G$ with $g(xw)=g(x)g|_x(w)$ for all $w\in X^\ast$; we call $g|_x$ the \emph{section} of $g$ at $x$. See~\cite{selfsimilar} for details.     Note that if one defines $g|_w$ for $w\in X^*$ inductively by $g|_{xv} = (g|_x)|_v$ for $w=xv$ with $x\in X$, then $g(uv) = g(u)g|_u(v)$ for all $u,v\in X^\ast$.   The rules $(gh)|_w= g|_{h(w)}h|_w$, $1|_w=1$ and $g^{-1}|_w = (g|_{g^{-1}(w)})^{-1}$ are easily verified for all $g,h\in G$ and $w\in X^\ast$.  If $A\subseteq G$ and $Y\subseteq X^\ast$, then we put $A|_Y =\{g|_w: g\in A, w\in Y\}$.

A length-preserving permutation is \emph{prefix-preserving} if it preserves the length of the longest common prefix of any two words.  Any self-similar group action is prefix-preserving and the group of all length-preserving and prefix-preserving permutations of $X^\ast$ is self-similar and can be identified with $\mathrm{Aut}(T_X)$ where $T_X$ is the Cayley graph of $X^\ast$, that is, the regular rooted $|X|$-ary tree.  Note that $\partial T_X$ can be identified with $X^{\omega}$ and so $\mathrm{Aut}(T_X)$ acts  on $X^\omega$.  The action of $g\in\mathrm{Aut}(T_X)$ on an infinite word is given by the formula $g(x_1x_2\cdots) = g(x_1)g|_{x_1}(x_2)g|_{x_1x_2}(x_3)\cdots$.  One can, in fact, identify $\mathrm{Aut}(T_X)$ with the isometry group of $X^{\omega}$ with respect to the metric $d(u,v) = 2^{-|u\wedge v|}$ where $u\wedge v$ is the longest common prefix of $u$ and $v$.  Note that $\mathrm{Aut}(T_X)$ is naturally a compact and totally disconnected group, i.e., a profinite group.  See~\cite{GNS} for details.

If $A\subseteq \mathrm{Aut}(T_X)$ is a subset that is closed under taking sections, then we can encode the action of $A$ on $X^\ast$ and $X^\omega$ into a \emph{state diagram} (also called Moore diagram). This is an edge-labeled  directed graph (digraph) with vertex set $A$ and with edges of the form \[a \xrightarrow{\,\,x\mid a(x)\,\,} a|_x\] for $a\in A$ and $x\in X$. The elements of $A$ are referred to as \emph{states} in this context. The action of a state $a$ on a finite or infinite word $w$ over $X$ can be computed from the state diagram by following the unique path from $a$ whose left hand side edge labels read the word $w$ and sending $w$ to the corresponding sequence of right hand side edge labels. The terminal state of the path is $a|_w$ if $w\in X^\ast$.

For example, consider the adding machine over $X=\{0,1\}$. It has states $A=\{a,1_{X^\ast}\}$, and the action on $X^\ast$ is given by  $a(0)=1$, $a(1)=0$ and sections $a|_0=1_{X^\ast}$, $a|_1=a$.  The state diagram is drawn in Figure~\ref{f:adding.machine}.
 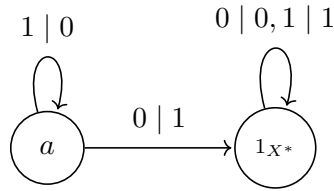
\begin{figure}[htbp]
\begin{center}
\begin{tikzpicture}[->,shorten >=1pt,%
auto,node distance=3cm,semithick,
inner sep=5pt,bend angle=30]
\tikzset{every loop/.style={min distance=10mm,looseness=10}}
%\tikzset{every state/.style={minimum size=30pt}}.
\node[state] (E) {\scriptsize{$1_{X^\ast}$}};
\node[state] (A) [left of=E] {$a$};
%\tikzstyle{every node}=[font=\footnotesize]
\path (A) edge [loop above] node [above]  {$1\mid 0$} (A)
      (A) edge [right] node [above]   {$0\mid 1$} (E)
      (E) edge [loop above] node [above] {$0\mid 0, 1\mid 1$} (E);
\end{tikzpicture}
\end{center}
\caption{State diagram for the adding machine~\label{f:adding.machine}}
\end{figure}

In computer science, finite state diagrams are used as visualizations of finite state automata, and so we shall call a finite subset $A\subseteq \mathrm{Aut}(T_X)$ that is closed under taking sections a (finite) \emph{automaton}, and will refer to the elements of $A$ as the states of $A$. More details, including the relationship with automata from computer science, can be found in~\cite{GNS}.
An automorphism $g\in \mathrm{Aut}(T_X)$ is said to be \emph{finite state} if it has only finitely many distinct sections;  equivalently, $g$ is finite state if it is a state of some finite automaton.  The finite state automorphisms form a countable self-similar subgroup of $\mathrm{Aut}(T_X)$. See~\cite{selfsimilar,GNS} for details.

If $G$ is any group, the subsets of $G$ form a monoid with involution under the product $CD=\{cd: c\in C,d\in D\}$ and involution $C^{-1} = \{c^{-1}:c\in C\}$, and hence we can talk about positive and negative powers of a subset of $G$ with respect to this structure.
If $A\subseteq \mathrm{Aut}(T_X)$ is an automaton, then so is $A^k$ for any $k\in \mathbb Z$, and hence the subgroup $G=\bigcup_{k\in \mathbb Z}A^k$ generated by $A$ is self-similar and consists entirely of finite state automorphisms.
Self-similar groups generated by automata are called \emph{automaton groups}; these were studied long before the general notion of a self-similar group was defined.  All automaton groups have decidable word problem~\cite{GNS} and one can in general ask algorithmic questions about them since they are given by a finite data structure: the state diagram of a generating automaton.

The notion of a contracting self-similar group is fundamental. We proceed with the definition and some basic facts that can be extracted from~\cite[Chapter~2.11]{selfsimilar}, sometimes in the body of the text.
A self-similar group $G$ over the alphabet $X$ is said to be \emph{contracting} if there exists a finite automaton $N \subseteq G$ such that, for every $g \in G$, there exists $n \in \mathbb N$ such that $g|_{X^n} \subseteq N$.   Since $N$ is closed under sections, if $g|_{X^n}\subseteq N$, then also $g|_{X^nX^*}\subseteq N$.  The smallest such set $N$ is then called the \emph{nucleus} of $G$.
To describe the nucleus we need a basic lemma about finite digraphs.

\begin{Lem}\label{l:essential}
Let $\Gamma$ be a finite digraph with $n$ vertices.  The following are equivalent for a vertex $v$ of $\Gamma$:
\begin{enumerate}
  \item $v$ is reachable from a strongly connected component of $\Gamma$ containing an edge;
  \item $v$ is reachable from a directed cycle in $\Gamma$;
  \item there is a left infinite directed path ending at $v$;
  \item for each $k\geq 0$, there is a directed path of length $k$ ending at $v$;
  \item there is a directed path of length $n$ ending at $v$;
  \item there is a directed path of length $k$ ending at $v$ for some $k \geq n$.
\end{enumerate}
\end{Lem}
\begin{proof}
Any strongly connected component containing an edge contains a directed cycle so $(1)\implies (2)$.  If $v$ can be reached from a directed cycle $c$, then $\cdots cccp$ has endpoint $v$ for some directed path $p$, whence $(2)\implies (3)$.  Trivially, $(3)\implies (4)\implies (5)\implies (6)$.  If $p$ is a directed path of length $k\geq n$, then $p$ must visit some vertex twice and hence have a subpath that is a directed cycle, which belongs to a strongly connected component of $\Gamma$ containing an edge.  Thus $(6)\implies (1)$.
\end{proof}

We call vertices satisfying the equivalent conditions of Lemma~\ref{l:essential} \emph{essential}. If $A$ is an automaton, then the \emph{essential states} of $A$ are the essential vertices of the state diagram of $A$.   From Lemma~\ref{l:essential}, one can algorithmically compute the essential vertices of a finite digraph. The motivation for our terminology comes from symbolic dynamics.  In that setting, a digraph is called \emph{essential} if each vertex is visited by a bi-infinite directed path or, equivalently, if it has no sources or sinks. Every finite digraph contains a unique largest essential subgraph~\cite[Proposition~2.2.10]{MarcusandLind}.  In a finite digraph with no sinks, like an automaton, the largest essential subgraph is precisely the subgraph induced by what we have called the essential vertices.
We remark that, as unlabeled digraphs, the state diagrams of $A$ and $A^{-1}$ are isomorphic via $a\mapsto a^{-1}$, but this mapping does not preserve the labels~\cite{GNS}.   Hence, the essential states of $A^{-1}$ are the inverses of the essential states of $A$.

An automaton $A$ is the nucleus of a self-similar group (namely, the one it generates) if and only if $A=A^{-1}$, $1_{X^\ast}\in A$,  $A|_X=A$ and $A^2|_{X^k}\subseteq A$  for some $k\geq 0$;  see~\cite[Lemma~2.11.2]{selfsimilar} and the discussion thereafter.  This is a decidable condition. One can algorithmically construct $A^{-1}$ and $A^2$ from $A$ and the first three conditions are easy to verify.  Let $B$ be the set of essential states of $A^2$.  We claim that in the presence of the first three conditions, the fourth condition is equivalent to $A=B$.  The condition that $A|_X=A$ ensures that every element of $A$ is the endpoint of a left infinite directed path in the state diagram of $A$, which is a subgraph of the state diagram of $A^2$ as $1_{X^\ast}\in A$, and so $A\subseteq B$ by Lemma~\ref{l:essential}(3). The fourth condition is equivalent to $B \subseteq A$:  if $A^2|_{X^k}\subseteq A$, then $B\subseteq A$ by Lemma~\ref{l:essential}(4); conversely, if $B\subseteq A$, then Lemma~\ref{l:essential}(5) shows that $A^2|_{X^n}=B\subseteq A$ where $n=|A^2|$.  Note that the condition $1_{X^\ast}\in A$ is redundant in the presence of the other conditions as $1_{X^\ast}$ is an essential state of $A^2$.

There is also a well-known procedure (implemented in the computer system GAP) to build the nucleus $N$ of the group $G$ generated by an automaton $A$, provided $G$ is contracting. We construct a sequence of automata as follows.   Put $A_0=A\cup A^{-1}$.   Assume that the automaton $A_n$ has been constructed so that $A_0\subseteq A_1\subseteq\cdots \subseteq A_n$, with $A_j=A_j^{-1}$ for $0\leq j\leq n$,  and each state in $A_n\setminus A_0$ belongs to $N$.  Build $A_{n+1}$ by adding to $A_n$ the essential states of $A_n^2\setminus A_n$; such states must belong to $N$ by the argument of the previous paragraph, and we have that $A_{n+1}=A_{n+1}^{-1}$.  Then $A_0\subseteq A_1\subseteq\cdots$ and this sequence stabilizes since the nucleus is finite.  If the sequence stabilizes at $A_k$, then $A_k$ contains the nucleus by~\cite[Lemma~2.11.2]{selfsimilar} and the nucleus $N$ will consist of the essential states of $A_k^2$.

A finite state automorphism $g\in \mathrm{Aut}(T_X)$ is said to be \emph{bounded} if there is a constant $C>0$ such that, for any $n\geq 0$, there are at most $C$ words $w\in X^n$ with $g|_w\neq 1$.  The bounded automorphisms form a self-similar subgroup of the group of finite state automorphisms.  An automaton generates a group of bounded automorphisms if and only if after removing the state of the identity function (if it is part of the automaton), each strongly connected component of the state diagram is a single vertex or a cycle, and no cycle can reach another cycle.  Such an automaton is called \emph{bounded}.  Any group generated by a bounded automaton is amenable~\cite{boundedaut}.  See~\cite{selfsimilar} for more on bounded automata.

A self-similar group $G$ over $X$ is called \emph{spherically transitive} if it acts transitively on $X^n$ for each $n\geq 0$.  This is equivalent to $G$ being ergodic with respect to the product of uniform measures on $X^{\omega}$ (cf.~\cite{GNS}).  The self-similar group $G$ is said to be \emph{self-replicating} if it is spherically transitive and $\psi_x\colon \mathrm{Stab}_G(x)\to G$ given by $\psi(x)=g|_x$ is surjective for all $x\in X$.   A self-replicating group is always infinite since it has a proper finite index subgroup mapping onto it.

Associated to any self-similar group $G$ over the alphabet $X$ is a left cancellative LCM monoid $M=X^\ast G$. (LCM means any two elements admitting a common right multiple admit a least common right multiple.)  The product in $M$ is given by $ug\cdot vh = ug(v)g|_vh$ where $u,v\in X^\ast$ and  $g,h\in G$.  The action of $M$ on the left of itself via multiplication is by injective mappings and hence we can form the inverse hull $S$ of $M$, which is the inverse monoid of partial injective mappings of $M$ generated by the left regular action of $M$; see~\cite{ExelSteinbergHull} for more details on inverse hulls of LCM monoids.  The inverse semigroup $S$ consists of $0$ (the empty map) and all elements of the form $ugv^{\ast}$ with $u,v\in X^\ast$ and $g\in G$ (where $v^\ast$ is the inverse of left multiplication by $v$ in $S$).  As a partial mapping on $M$, the domain of $ugv^\ast$ is $vM$, its range is $uM$ and its action is $vm\mapsto ugm$.  Recall that a $\ast$-semigroup is a semigroup with an involution $\ast$ satisfying $(st)^\ast =t^\ast s^\ast$.  One can present $S$ as a $\ast$-semigroup  by the generating set $X\cup G$ and the relations $g^\ast =g^{-1}$, $g\cdot h=gh$, $gx=g(x)g|_x$ and $x^\ast y=\delta_{x,y}$ for $g,h\in G$  and $x,y\in X$.  It is well known that the inverse semigroup $S$ is congruence-free~\cite[Proposition~6.2]{LawsonCorrespond}.
The non-zero idempotents of $S$ are the elements of the form $ww^\ast$ with $w\in X^\ast$ and $ww^\ast\leq uu^\ast$ if and only if $u$ is a prefix of $w$.  Moreover, $ww^\ast vv^\ast=0$ whenever $v,w$ are not prefix comparable.  From this, it immediately follows that $\{xx^\ast: x\in X\}$ covers $1$.

There are also natural faithful actions of $S$ on $X^\ast$ and $X^\omega$.  Namely, $ugv^\ast\colon vX^\ast\to uX^\ast$   and $ugv^\ast\colon vX^\omega\to uX^\omega$ are given by $vw\mapsto ug(w)$ for $w$ in $X^\ast$ or  $X^\omega$.
We denote the action of $s$ on a (finite or infinite) word $w$ in its domain by $s(w)$.

 Note that the action of $S$ on $X^\ast$ can be identified with its action on principal filters (i.e., its Munn representation) and its action on $X^\omega$ can be identified with its action on tight filters in the sense of Exel~\cite{Exel,ExPadKatsura}.  Hence the groupoid of germs for the action of $S$ on $X^\omega$ is isomorphic to its tight groupoid $\G_T(S)$.  This groupoid is always minimal and effective~\cite{ExPadKatsura} but is not always Hausdorff.  We will not use ample groupoids explicitly in this paper other than to discuss their Hausdorffness, but they lurk in the background and many of our ideas can be translated into that language.  We do mention in passing that it is shown in~\cite{ExPadKatsura} that if the self-similar group $G$ is amenable, then so is the groupoid $\G_T(S)$. It is shown in~\cite[Theorem~5.6]{Nekcstar} that $\G_T(S)$ is amenable when $G$ is contracting and self-replicating; note that Nekrashevych assumes in that section of his paper that the groupoid is Hausdorff, but this is not needed for the amenability result, which relies on the polynomial growth of orbits for a self-replicating, contracting group.  When $\G_T(S)$ is amenable, its reduced $C^\ast$-algebra is the same as its universal $C^\ast$-algebra, which is the Nekrashevych $C^\ast$-algebra~\cite{Nekcstar}. It follows from the results of~\cite{nonhausdorffsimple} that if a minimal, effective, amenable and second countable groupoid has a non-simple complex Steinberg algebra, then its $C^\ast$-algebra is not simple.

We remark that it is easy to see from the nucleus whether the groupoid $\G_T(S)$ associated to a contracting self-similar group $G$ is Hausdorff. Let $G$ be a contracting self-similar group with nucleus $N$, and consider the state diagram of the automaton $N$. First remove all edges from the state diagram of $N$ except those with labels of the form $x \mid x$, and then take the subgraph of this graph induced by those vertices from which $1$ is reachable by a directed path. Denote the edge-labeled digraph obtained in this way by $\mathcal H$. For simplicity, we replace the labels $x\mid x$ in $\mathcal H$ by just $x$. Note that all letters $x \in X$ label loops in $\mathcal H$ at $1$.

\begin{Prop}
\label{p:Hd}
The  groupoid $\G_T(S)$ associated to the contracting group $G$ is Hausdorff if and only if the only directed cycles of $\mathcal H$ are the loops at the vertex $1$.  In particular, given as input the nucleus of a contracting self-similar group, one can algorithmically decide if the associated groupoid is Hausdorff.
\end{Prop}
\begin{proof}
By~\cite[Theorem~12.2]{ExPadKatsura} (see also~\cite[Lemma~5.4]{Nekcstar}), $\G_T(S)$ is Hausdorff if and only if for each $g \in G$, there exists a finite set $F_g$ such that $\{w \in X^\ast: gw=w\}=F_gX^\ast$ (where $gw$ is the product in the monoid $M=X^{\ast} G$). If $gw=w$, we say that $w$ is \emph{strongly fixed} by $g$, and this holds if and only if $g(w)=w$ and $g|_w=1$, or equivalently, if $g$ fixes $wX^\ast$. In this case, $g$ of course strongly fixes any word in $wX^\ast$.

Let $N$ be the nucleus of $G$.  First, we prove that if such a finite set $F_n$ exists for all $n \in N$, then there exists one for any $g \in G$. Let $k \in \mathbb N$ such that $g|_{X^k} \subseteq N$, and take a word $w$ of length at least $k$. Then $w=uv$ with $|u|=k$, and $gw=guv=g(u)g|_uv$, where $g|_u \in N$. Thus $gw=w$ holds if and only if $g(u)=u$ and $v \in F_{g|_u}X^\ast$, that is, if and only if $w \in uF_{g|_u}X^\ast$.
Let
\[F_g=\{w \in X^\ast: |w| < k \hbox{ and } gw=w\} \cup \{uF_{g|_u} : |u| =k \hbox{ and } g(u)=u \}.\]
The above set is finite, and we have $\{w \in X^\ast: gw=w\}=F_gX^\ast$ by the previous observation.

We now proceed to prove the equivalence of the statement and the existence of such sets $F_n$ for all $n \in N$. Let $w\in X^\ast$. As observed, we have $nw=w$ if and only if $n(w)=w$ and $n|_w=1$, that is, if and only if there is a path in $\mathcal H$ from $n$ to $1$ labeled by $w$.
Thus the words strongly fixed by $n$ are precisely the labels of path in $\mathcal H$ from $n$ to $1$.

Note that since all letters label loops at $1$, if there is a directed cycle in $\mathcal H$ that is not a loop at $1$, then it is based at some $n\in N\setminus \{1\}$ and never passes through $1$.  Say that $x$ labels this cycle and let $w$ label a simple path from $n$ to $1$ in $\mathcal H$.  Suppose that the set of words strongly fixed by $n$ is $F_nX^*$.   Note that $x^mw$, for all $m \in \mathbb N$, labels a path from $n$ to $1$ no subpath of which reaches $1$.   Thus the words $x^mw$ are all strongly fixed by $n$, and additionally they all have no proper prefix strongly fixed by $n$.  Thus  $x^mw\in F_nX^\ast$ implies $x^mw\in F_n$ for all $m\geq 1$.  Therefore, no finite $F_n$ can satisfy the condition required.

Conversely, assume $\mathcal H$ contains no directed cycle except the loops at $1$ and let $n \in N$. If no words are strongly fixed by $n$, then $F_n=\emptyset$ suffices. Otherwise, $n$ is a vertex in $\mathcal H$. Consider all the distinct simple directed paths from $n$ to $1$ in $\mathcal H$; as $\mathcal H$ is finite, there are only finitely many, denote their labels by $w_1, \ldots, w_k$. As the only cycles in $\mathcal H$ are the loops at $1$, any path from $n$ to $1$ in $\mathcal H$ is a simple path followed by any number of such loops. Thus the set of the labels of  paths from $n$ to $1$ (i.e., of strongly fixed words) is exactly $\{w_1, \ldots, w_k\}X^\ast$, and so $F_n=\{w_1, \ldots, w_k\}$ suffices.
\end{proof}

\section{Nekrashevych algebras}\label{s:na}
Recall that a $\ast$-algebra over a field $K$ with an automorphism $\sigma$ satisfying $\sigma^2=1$ is a $K$-algebra $A$ with an involution $\ast$ satisfying $(ab)^\ast =b^\ast a^\ast$, $(a^\ast)^\ast=a$ and $(ca)^\ast = \sigma(c)a^\ast$.
A contracted inverse semigroup algebra $K_0S$ is always a $\ast$-algebra via $\left(\sum_{s\in \S} c_ss\right)^\ast = \sum_{s\in \S} \sigma(c_s)s^\ast$.  From now on we tacitly assume that $\sigma$ is the identity (although it is usual to use complex conjugation over $\mathbb C$).

Let $G$ be a self-similar group over the finite alphabet $X$.
Nekrashevych originally introduced an algebra associated to $G$ and $X$ in the $C^\ast$-algebra setting~\cite{Nekcstar}, but then later studied the algebraic version~\cite{Nekrashevychgpd}.  Further work on Nekrashevych algebras can be found in~\cite{nonhausdorffsimple,ExPadKatsura,simplicity}.

The \emph{Nekrashevych algebra} $\N{K}{G}{X}$ of $(G,X)$ with coefficients in $K$ is the $\ast$-algebra over $K$ with generating set $X\cup G$ and relations:
\begin{itemize}
\item [(G)] $g^\ast =g^{-1}$ and $g\cdot h=gh$ for $g,h\in G$;
\item [(SS)] $gx = g(x)g|_x$ for $g\in G$, $x\in X$;
\item [(CK1)] $y^\ast x=\delta_{x,y}$ for $x,y\in X$;
\item [(CK2)] $\sum_{x\in X} xx^\ast=1$.
\end{itemize}

Observe that (SS) and (CK2), in the presence of (CK1), can be replaced by the single family of relations:
\begin{itemize}
\item [(CKSS)] $g=\sum_{x\in X}g(x)g|_x x^\ast$, for $g\in G$.
\end{itemize}
Indeed, (CK2) is the special case of (CKSS) with $g=1$ and
(SS) follows from computing $gx$ using (CK1) and (CKSS).  Conversely, assuming (SS) and (CK2), we have that $g= g\sum_{x\in X} xx^\ast = \sum_{x\in X}g(x)g|_xx^\ast$.

It is well known that (CK2) implies the relation $1=\sum_{w\in X^k}ww^\ast$ for any $k\geq 0$ by a straightforward induction argument. For if $1=\sum_{u\in X^k}uu^\ast$, then \[1=\sum_{u\in X^k}uu^\ast=\sum_{u\in X^k}u\left(\sum_{x\in X}xx^\ast\right) u^\ast = \sum_{w\in X^{k+1}}ww^\ast.\]

Notice that (G), (SS) and (CK1) define the contracted semigroup algebra $K_0S$ where $S$ is the associated inverse semigroup from above.  The ideal $I$ of $K_0S$ generated by the element $1-\sum_{x\in X}xx^\ast$ is called the \emph{Cuntz-Krieger ideal}. The Nekrashevych algebra is then $K_0S/I$.

Notice that when $G$ is trivial, one obtains the classical Leavitt algebra, which is simple~\cite{Leavitt,LeavittBook}.  It is thus natural to investigate for which self-similar groups the Nekrashevych algebra is simple.

We remark that $\N{K}{G}{X}$ is naturally a $\mathbb Z$-graded algebra where the homogeneous component of degree $n$ is spanned by the $ugv^\ast$ with $|u|-|v|=n$. The degree zero component of this algebra is studied in~\cite{Nekrashevychgpd}.  It can be viewed as the direct limit $\varinjlim M_{X^n}(KG)$ via the homomorphisms $M_{X^n}(KG)\to M_{X^{n+1}}(KG)$ given by \[gE_{u,v}\longmapsto \sum_{x\in X}g|_xE_{ug(x),vx}\] where $u,v\in X^n$ and $E_{w,z}$ is the elementary matrix unit indexed by the words $w,z$.  This directed system consists of surjective homomorphisms when $G$ is self-replicating, cf.~\cite{Nekrashevychgpd}.  The isomorphism takes $ugv^\ast$ to $gE_{u,v}$. See~\cite{Nekrashevychgpd} for details.

Nekrashevych proved in the $C^\ast$-algebra context that the algebra associated to a contracting group is finitely presented~\cite[Theorem~4.2]{Nekcstar}.  We include his result here,  with a more detailed proof, for the algebraic setting.

\begin{Thm}[Nekrashevych]
Let $G$ be a contracting self-similar group over the alphabet $X$ with nucleus $N$.  Then $\N{K}{G}{X}$ is the $\ast$-algebra over $K$ with generators $X\cup N$ and the following defining relations.
\begin{itemize}
\item [(N)] $n^\ast =n^{-1}$ for $n\in N$ and $n\cdot n'=nn'$ if $n,n',nn'\in N$;
\item [(CK1)] $y^\ast x=\delta_{x,y}$ for $x,y\in X$;
\item [(CKSN)] $n=\sum_{x\in X} n(x)n|_xx^\ast$ for $n\in N$.
\end{itemize}
In particular, $\N{K}{G}{X}$ is finitely presented.
\end{Thm}
\begin{proof}
For the purposes of this proof we recall that $1\in N=N^{-1}$ and that $N$ is closed under sections.  Let  $A$ be the $\ast$-algebra over $K$ with generators $X\cup N$ defined by the relations (N), (CK1) and (CKSN).

If $g\in G$ and $m\geq 0$ is such that $g|_{X^m}\in N$, then in $\N{K}{G}{X}$
\begin{equation}\label{eq:express.g}
g = g\sum_{w\in X^m}ww^\ast = \sum_{w\in X^m}g(w)g|_ww^*
\end{equation}
 by (CK2) and (SS), and $g|_w\in N$ for each $w\in X^m$.  Thus $\N{K}{G}{X}$ is generated by $X$ and $N$ as a $\ast$-algebra over $K$.
Clearly (N) is a consequence of (G).  Note that (CKSN) follows from (CKSS).  Thus $\N{K}{G}{X}$ satisfies the relations (N), (CK1) and (CKSN) and so there is a natural surjective homomorphism $A\to \N{K}{G}{X}$ that is the identity on $N$ and $X$.

For $g\in G$ and $m\geq 0$ with $g|_{X^m}\subseteq N$, define a map $\varphi\colon G\to A$ by  $\varphi(g)=\sum_{w\in X^m} g(w)g|_ww^\ast$.  Note that this element is independent of the choice of $m$, for
\begin{align*}
\sum_{w\in X^m} g(w)g|_ww^\ast&=\sum_{w\in X^m} g(w)\sum_{x\in X}g|_w(x)(g|_w)|_xx^\ast w^\ast\\ &=\sum_{w\in X^m}\sum_{x\in X} g(w)g|_w(x)(g|_w)|_xx^\ast w^\ast= \sum_{v\in X^{m+1}}g(v)g|_vv^\ast
\end{align*}
by (CKSN) and (N) as $g|_w\in N$ for all $w\in X^m$.  Notice that $\varphi$ is the identity on $N$ (by taking $m=0$).
We extend $\varphi$ to $G\cup X$ by sending $X$ to itself via the identity map; we check that this induces a well-defined homomorphism from $\N{K}{G}{X}$ to $A$ by verifying that the relations of $\N{K}{G}{X}$ are preserved.  Then $\varphi$ will be inverse to the homomorphism of the previous paragraph by \eqref{eq:express.g}.

Clearly (CK1) is preserved.  Notice that $1\in N$, and so (CKSN) implies (CK2) by taking $n=1$.
We now check that (G) is preserved, that is, $\varphi$ is a $\ast$-homomorphism on $G$.  First of all, if $g|_{X^m}\subseteq N$, then also $g^{-1}|_{X^m}\subseteq N$ and $\varphi(g)=\sum_{w\in X^m}g(w)g|_ww^\ast$, and so using (N), we have that
\begin{align*}
(\varphi(g))^\ast &= \sum_{w\in X^m}w(g|_w)^{-1}g(w)^\ast =\sum_{w\in X^m}w(g^{-1})|_{g(w)}g(w)^\ast\\
& =\sum_{v\in X^m}g^{-1}(v)(g^{-1})|_vv^\ast=\varphi(g^{-1}).
\end{align*}

Next let $g,h\in G$ and choose $m\geq 0$ so that $g|_{X^m}$, $h|_{X^m}$ and $(gh)|_{X^m}$ all belong to $N$.  Then
\begin{equation*}
\varphi(g) \varphi(h) =\sum_{w\in X^m}g(w)g|_ww^\ast\cdot \sum_{v\in X^m}h(v)h|_vv^\ast = \sum_{v\in X^m}      g(h(v))g|_{h(v)}h|_vv^\ast
\end{equation*}
But $(gh)|_v = g|_{h(v)} h|_v$ and since $g|_{h(v)},h|_v, (gh)|_v\in N$, we may use (N) to rewrite the right hand side of the above equation as
$\sum_{v\in X^m}  g(h(v))(gh)|_vv^\ast= \varphi(gh)$, verifying (G).

Finally, we verify (SS).  Let $g\in G$ and choose $m\geq 1$ so that $g|_{X^m}\subseteq N$.  Note that, for $x\in X$, we have that $(g|_x)|_{X^{m-1}}\subseteq N$.  Then
\begin{align*}
\varphi(gx) &= \sum_{w\in X^m}g(w)g|_w w^\ast x  = \sum_{u\in X^{m-1}}g(xu)g|_{xu}u^\ast
= \sum_{u\in X^{m-1}}g(x)g|_x(u)(g|_x)|_uu^\ast\\ & = g(x)\sum_{u\in X^{m-1}}g|_x(u)(g|_x)|_u u^\ast = \varphi(g(x) g|_x).
\end{align*}

Thus $A$ satisfies (SS) and we conclude $A\cong \N{K}{G}{X}$, as required.
\end{proof}

In particular, $\N{K}{G}{X}=\N{K}{\langle N\rangle}{X}$.  Since it is decidable if an automaton is a nucleus of a self-similar group and the algebra depends only on the nucleus, for algorithmic problems concerning Nekrashevych algebras  of contracting groups it is best to take as input the nucleus.  Moreover, the nucleus of a contracting group can be computed from any automaton generating the group, as we saw earlier.

\section{Simplicity of Nekrashevych algebras}\label{s:simple}
Our goal is to use Theorem~\ref{t:cong-free.simple} to study simplicity of Nekrashevych algebras.  For this section $G$ will always be a self-similar group over the finite alphabet $X$ and $S$ will be the associated inverse semigroup.

\subsection{The tight and singular ideals}
The next proposition identifies the Cuntz-Krieger ideal of $K_0S$ as  the tight ideal $\tightid{K}{S}$, and thus
$\tight S\cong K_0S/\tightid K S \cong \N K G X$.  This also follows from~\cite[Section~6.3]{gridealSteinberg}, which proves the corresponding statement for the more general class of Exel-Pardo algebras.  Our result provides some additional information that gives a hands-on description of $\tightid K S$.

\begin{Prop}
\label{p:tightideal}
Let $G$ be a self-similar group over the finite alphabet $X$ and $S$ the associated inverse semigroup. Then the following are equivalent for any $a \in K_0S$:
\begin{enumerate}
\item $a$ is in the Cuntz-Krieger ideal;
\item $a \in \tightid{K}{S}$;
\item any $\ol w \in X^\omega$ has a prefix $w \in X^\ast$ with $aw=0$;
\item there exists $N\geq 0$ with $aX^N=0$.
\end{enumerate}
\end{Prop}
\begin{Proof}
We begin by showing that (3) and (4) are equivalent. The implication $(4) \implies (3)$ is trivial. For the converse, consider the subgraph $T$ of the tree $T_X$ induced on the vertices $w$ with $aw \neq 0$. As $aw \neq 0$ implies $au \neq 0$ for any prefix $u$ of $w$, we have that $T$ is a subtree containing $1$, and it is locally finite since $X$ is finite. If $T$ is infinite, that is, if (4) is not satisfied, then, by K\H onig's lemma, $T$ contains a right infinite simple path starting at $1$ whose label is a word $\ol w \in X^\omega$, all of whose prefixes $w$ belong to $T$ and hence satisfy $aw\neq 0$.

We proceed by showing $(1)  \implies (2) \implies (4)  \implies (1)$.
Recall that the Cuntz-Krieger ideal is generated by the element $1 - \sum_{x \in X}xx^\ast$. This is also a generator of $\tightid K S$, as the idempotents $\{xx^\ast: x \in X\}$ cover $1$ and are pairwise orthogonal,  whence $1-\sum_{x \in X}xx^\ast = \prod_{x \in X} (1-xx^\ast)\in \tightid K S$. This shows that $(1) \implies (2)$.

We next show that the elements satisfying (4) form an ideal. If $a,b \in K_0S$ are such that $aX^{N}=bX^{L}=0$, then $(a+b)X^{N+L}=0$.  Also, for any $c \in K_0S$, $caX^{N}=0$, and so the set of elements satisfying (4) form a left ideal. What remains to show is that if $aX^N=0$, then, for any $s \in S$, $asX^L=0$ for some $L$. Put $s=ugv^\ast$, and let $L=N+|v|$. Then, for  $w \in X^L$, either $v^\ast w=0$, and so $asw=0$, or $w=vz$ with $z \in X^{N}$, in which case $asw=augz=aug(z)g|_z=0$,
as $|ug(z)|=|uz|\geq N$.

It remains to show that the generators of $\tightid K S$ satisfy (4). Let $v \in X^\ast$ and $F$ a finite cover of $vv^\ast$. Put $N=\max\{|u|: uu^\ast \in F\}$ and let $w \in X^N$. If $w \notin vX^\ast$, then $vv^\ast w=0$, and so $uu^\ast w=0$ for all $u\in F$, whence $\prod_{u \in F}(vv^\ast-uu^\ast)w=0$. If $w \in vX^\ast$, then $ww^\ast \leq vv^\ast$, and so there exists $u \in F$ with $uu^\ast ww^\ast \neq 0$. Since $|w| \geq |u|$, we must have $uu^\ast ww^\ast=ww^\ast$.
Then $(vv^\ast -uu^\ast)w=(vv^\ast-uu^\ast)ww^\ast w=(ww^\ast-ww^\ast)w=0$, and so
 $\prod_{u \in F}(vv^\ast-uu^\ast)w=0$. This proves $(2) \implies (4)$.

For the last part, assume now that $aX^N=0$ for some $N\geq 0$. Then since $1-\sum_{w \in X^{N}}ww^\ast$ is in the Cuntz-Krieger ideal, so is
$a=a(1-\sum_{w \in X^{N}}ww^\ast)$. This shows that $(4) \implies (1)$, completing the proof.
\end{Proof}

In other words, $a\in K_0S$ belongs to the Cuntz-Krieger ideal if and only if every infinite word has a prefix $w$ with $aw=0$.  We say that an infinite word $\ol w$ witnesses that $a\notin \tightid{K}{S}$ if $aw\neq 0$ for all prefixes $w$ of $\ol w$.   Note that an element $a\in K_0S$ is singular if and only if, for each $u\in X^\ast$, there exists $v\in X^\ast$ with $auv=0$.  Indeed, if $uu^\ast\in E(S)\setminus \{0\}$, then $ww^\ast\leq uu^\ast$ if and only if $w=uv$ with $v\in X^\ast$, and $auv(uv)^\ast =0$ if and only if $auv=0$.  This characterization of the singular ideal, in conjunction with Proposition~\ref{p:tightideal}(4), provides an alternate argument that $\tightid K S\subseteq \sing{K}{S}$  in this setting.
Thus Theorem~\ref{t:cong-free.simple} has the following corollary for Nekrashevych algebras, in light of Proposition~\ref{p:tightideal}.

\begin{Thm}
\label{t:simplicity}
Let $G$ be a self-similar group over a finite alphabet $X$ and $K$ a field.  Let $S$ be the associated inverse semigroup. Then $\N{K}{G}{X}=K_0S/\tightid K S$ has a unique maximal ideal consisting of the cosets of the elements of \[\sing K S =\{a\in K_0S: \forall u\in X^\ast, \exists v\in X^\ast, auv=0\}.\]  In particular, $\N{K}{G}{X}$ is simple if and only if $\tightid K S=\sing K S$.
\end{Thm}

Although the most useful characterization of the ideal $\sing K S /\tightid K S$ is the one above, it can also be described in a more natural way. The idea is reminiscent of the description of the unique simple quotient of the Nekrashevych $C^\ast$-algebra found in~\cite{Nekcstar}.

The action of $S$ on $X^\omega$ by partial maps can be extended naturally to a representation of $K_0S$ on the $K$-vector space $KX^\omega$ with basis $X^\omega$. (A previously undefined action of $s \in S$ on $w \in X^\omega$ is defined as $0$.) We denote the action of $a \in K_0S$ on $b \in KX^{\omega}$ by $ab$.
Notice that $\sum_{x \in X} xx^\ast$ fixes any infinite word $\ol w$, as denoting its first letter by $x_1$ we have
\[\big(\sum_{x \in X} xx^\ast\big)\ol w=x_1x_1^\ast(\ol w)=\ol w.\]
Denote the kernel of the representation by $I$. It then follows that $I$ contains the Cuntz-Krieger ideal, therefore we also obtain a representation of $\N K G X$ on $KX^\omega$.

\begin{Thm}
Let $G$ be a self-similar group over a finite alphabet $X$ and $K$ a field, and let $S$ be the associated inverse semigroup. The kernel $I$ of the representation of $K_0S$ on $KX^\omega$ is $\sing K S$. Consequently, the kernel of the action of $\N K G X$ on $KX^\omega$ is $\sing K S/\tightid K S$. Therefore $\N K G X$ is simple if and only if its representation on $KX^\omega$ is faithful.
\end{Thm}

\begin{Proof}
Since $\tightid K S \subseteq I$, by Theorem~\ref{t:simplicity} we obtain $I \subseteq \sing K S$. For the reverse inclusion, let $a=\sum_{s \in S} a_s s \in \sing K S$, and let $\ol w \in X^\omega$. We will show that $a \ol w=0$. Note that $a\ol w=0$ if and only if
\begin{equation}
\label{eq:bigger}
\sum_{s(\ol w)=t( \ol w)}a_s=0
\end{equation}
 for all $t \in \supp a$ with $t( \ol w) \neq 0$.

Denote the length $n$ prefix of $\ol w$ by $w_n$, and let $k$ be such that $k \geq |v|$ for any $ugv^\ast \in \supp a$. For any $ugv^\ast \in \supp a$ we have $(ugv^\ast)(\ol w) \neq 0$ if and only if $v$ is a prefix of $w_k$, if and only if $(ugv^\ast)(w_n) \neq 0$ for some $n\geq k$, which is if and only if $(ugv^\ast)(w_n) \neq 0$ for any $n \geq k$.

Notice that for any $s=ugv^\ast \in S$ and $z \in X^\ast$ we have $svz=ugz=ug(z)g|_z=s(vz)g|_z.$
So for any $w'\in X^\ast \cup X^\omega$, we have
\begin{equation}
\label{e:suffix}
s(vzw')=(svz)(w')=s(vz)g|_z(w').
\end{equation}
Thus if $w \in X^\ast$ is such that $s(w) \neq 0$ (i.e., $w$ is of the form $vz$), and $w'$ is a finite word, then $s (ww')$ is of the form $s(w) w''$ with $|w''|=|w'|$. It follows that if $s,t\in S$ are such that $s(w),t(w)\neq 0$ and  $s(ww')=t (ww')$, then $s(w)=t(w)$.

By~\eqref{e:suffix}, the action of $s=ugv^\ast$ on an infinite word $vzx_1x_2 \cdots$, where $z \in X^\ast, x_i \in X$, is given by
\[s(vzx_1x_2 \cdots)=s(vz)g|_{z}(x_1)g|_{zx_1}(x_2)g|_{zx_1x_2}(x_3) \cdots,\]
whereas
\[s(vzx_1 \cdots x_n)=s(vz)g|_{z}(x_1)g|_{zx_1}(x_2) \cdots g|_{zx_1 \cdots x_{n-1}}(x_n).\]
Thus if $w$ is such that $s(w)\neq 0$, then there exist $y_i \in X$ with
\[s(wx_1x_2 \cdots)=s(w)y_1y_2 y_3\cdots\] and
\[s(wx_1 \cdots x_n)=s(w)y_1 \cdots y_{n}.\]
It follows that if $s,t \in S$ are such that $s(w),t(w)\neq 0$ and $s(wx_1\cdots x_n)=t(wx_1\cdots x_n)$ for all $n$, then $s(wx_1x_2 \cdots)=t(wx_1x_2 \cdots)$.

Let $A=\{s \in \supp a: s(w_k) \neq 0\}$. For $n \geq k$ define the equivalence relation $\equiv_n$ on $A$ by $s \equiv_n t$ whenever $s (w_n)=t (w_n) $, and put $s \equiv_\omega t$ whenever $s(\ol w)=t(\ol w)$. Notice that $s\equiv_{n+1} t$ implies $s\equiv_n t$ by the observation following~\eqref{e:suffix}.
Since $A$ is finite, the decreasing chain of equivalences $\equiv_k\ \supseteq\  \equiv_{k+1}\ \supseteq \ldots$ eventually stabilizes, say at some index $N$. Then $\equiv_N$ implies $s(w_n)=t(w_n)$ for all $n \geq N$. Hence $s(\ol w)=t(\ol w)$ by the end if the previous paragraph (with the choice of $w=w_N$), thus $\equiv_N$ is contained in $\equiv_\omega$.

Since $a \in \sing K S$, there exists a word $u \in X^\ast$ with $aw_Nu=0$. Note that for any $t \in \supp a$, we have $tw_Nu \neq 0$ if and only if $t \in A$.
Thus from $aw_Nu=0$ we obtain that
\begin{equation}
\label{eq:small}
\sum_{sw_Nu=tw_Nu}a_s=0
\end{equation}
holds for all $t \in A$.
Consider the equivalence relation on $A$ defined by $s \equiv t$ if and only if $sw_Nu=tw_Nu$.
Note that if $s \equiv t$, then $s (w_Nu)=(sw_Nu)( \epsilon)=(tw_Nu) ( \epsilon)=t (w_Nu)$ and hence, by the observation after \eqref{e:suffix}, $s (w_N)=t(w_N)$, that is, $s \equiv_N t$, and therefore $s \equiv_\omega t$. Thus the $\equiv_\omega$-classes are unions of $\equiv$-classes. By~\eqref{eq:small}, we have that the sum of the coefficients $a_s$ is $0$ on each $\equiv$-class, and thus they must sum to $0$ on each $\equiv_\omega$-class as well, implying~\eqref{eq:bigger}. This proves $a \ol w=0$. The last two statements are immediate.
\end{Proof}

\begin{Rem}
Self-similar groups and their associated inverse semigroups $S$ can be analogously defined over infinite alphabets, see~\cite[Section 6]{simplicity} for details. In this case, $\tightid K S=0$, and so the tight algebra $K_0S/\tightid K S$ coincides with $K_0S$, and is simple if and only if $\sing K S=0$. The above proof never uses the finiteness of $X$ and thus shows that the kernel of the action of $K_0S$ on $KX^\omega$ is $\sing K S$ for any nontrivial alphabet $X$. In particular $K_0S$ is simple if and only if its representation on $KX^\omega$ is faithful.
\end{Rem}

We move on to show that if $\sing K S\setminus \tightid K S$ is non-empty, then it must intersect the group algebra of $G$, and in the contracting case it must contain an element supported on the nucleus.

The following proposition is immediate from
~\cite[Propostion~6.4]{simplicity},
together with the observation that $\tightid{K}{S}$ is an ideal.

\begin{Prop}
\label{p:KMtoKG}
Put $M=X^\ast G \leq S$, and let $a \in KM$. Then $a$ is uniquely of the form $\sum_{u \in X^\ast} ua_u$ with $a_u \in KG$, and $a$ is singular if and only if each $a_u$ is singular.  In particular, if $a \in \sing{K}{S} \setminus \tightid{K}{S}$, then  $a_u \in KG \cap \sing{K}{S} \setminus \tightid{K}{S}$ for some $u$.
\end{Prop}

Our next  proposition shows that when searching for singular elements that do not belong to the tight ideal, it suffices to look inside the group algebra.

\begin{Prop}
\label{p:supportinG}
Suppose $\sing{K}{S} \setminus \tightid{K}{S} \neq \emptyset$. Then $\sing{K}{S} \setminus \tightid{K}{S}$ intersects the group algebra $KG$.
\end{Prop}

\begin{Proof}
We begin by proving that $\sing{K}{S} \setminus \tightid{K}{S}$ intersects the $K$-algebra of the monoid $M=X^\ast G$.
Let $a \in \sing{K}{S} \setminus \tightid{K}{S} \neq \emptyset$, and let $\ol w$ be an infinite word which witnesses $a \notin \tightid{K}{S}$ by Proposition~\ref{p:tightideal}. Then for any prefix $w$ of $\ol w$, $aw \in \sing{K}{S} \setminus \tightid{K}{S}$ as well, as $aw \notin \tightid{K}{S}$ is witnessed by the word $w^\ast \ol w$.

Let $s=ugv^\ast \in \S$ and let $w \in X^\ast$ be a word longer than $v$. Then $ugv^\ast w=0$ if $v$ is not a prefix of $w$, and if $w=vw'$, then $ugv^\ast w=ugw'=ug(w')g|_{w'}$. In both cases, $ugv^\ast w \in M \cup \{0\}$.

Choose a prefix of $\ol w$ longer than any word $v$ with $ugv^\ast$ in $\supp a$. Then $aw \in KM$ and, as observed, $aw \in \sing{K}{S} \setminus \tightid{K}{S}$. Put $b=aw$, and apply Proposition~\ref{p:KMtoKG} to $b$. This guarantees an element $b_u \in KG \cap \sing{K}{S} \setminus \tightid{K}{S}$, which proves the claim.
\end{Proof}

We now turn to the case of a contracting group.

\begin{Prop}
\label{p:supportinN}
Suppose that the self-similar group $G$ is contracting with nucleus $N$. If $\sing{K}{S} \setminus \tightid{K}{S} \neq \emptyset$, then $\sing{K}{S} \setminus \tightid{K}{S}$ intersects $KN$.
\end{Prop}
\begin{Proof}
If $\sing{K}{S} \setminus \tightid{K}{S} \neq \emptyset$, then by Proposition~\ref{p:supportinG} there is an element $a \in KG \cap \sing{K}{S} \setminus \tightid{K}{S}$, witnessed by an infinite word $\ol w$. Let $n \in \mathbb N$ be such that for any $g \in \supp a$,
$g|_{X^n} \subseteq N$, and take the prefix $w_n$ of $\ol w$ of length $n$. Then $\supp(aw_n) \subseteq X^\ast N$, and $aw_n \in \sing{K}{S} \setminus \tightid{K}{S}$, as witnessed by $w_n^\ast \ol w$. Apply Proposition~\ref{p:KMtoKG} to $b=aw_n$ to obtain an element $b_u \in \sing{K}{S} \setminus \tightid{K}{S}$ with $u \in X^\ast$. Notice that $\supp b_u \subseteq N$ by construction.
\end{Proof}

Proposition~\ref{p:supportinN} leads to the following criterion for simplicity for contracting groups.

\begin{Cor}
\label{c:simplereducedtoKN}
Let $G$ be a contracting self-similar group acting over a finite alphabet $X$ with nucleus $N$ and $S$ the associated inverse semigroup. Then the Nekrashevych algebra $\N{K}{G}{X}$ is simple if and only if $KN \cap \sing{K}{S} \setminus \tightid{K}{S}$ is empty.
\end{Cor}
\begin{Proof}
By Theorem~\ref{t:simplicity}, $\N{K}{G}{X}$ is simple if and only if $\tightid K S=\sing K S$. By Proposition~\ref{p:supportinN}, $\tightid K S \neq \sing K S$ if and only if there is an element $a \in \sing K S \setminus \tightid K S$ supported on $N$.
\end{Proof}

\begin{Rem}
Proposition~\ref{p:supportinG} implies that $\N{K}{G}{X}$ is simple if and only if the homogenous component of degree $0$ under the natural $\mathbb Z$-grading is simple.  We sketch the argument, which requires some familiarity with Steinberg algebras of groupoids and tight groupoids of inverse semigroups.  Let $S_0$ be the inverse subsemigroup of $S$ consisting of $0$ and all $ugv^\ast$ with $|u|=|v|$.  Then $S_0$ contains all the idempotents of $S$.  Hence $\sing{K}{S_0} = \sing{K}{S}\cap K_0S_0$ and $\tightid{K}{S_0}\subseteq \tightid{K}{S}\cap K_0S_0$.  In fact, $\tightid{K}{S_0}=\tightid{K}{S}\cap K_0S_0$ since if $a\in \tightid{K}{S}\cap K_0S_0$ and $aX^N=0$, then  $a=a(1-\sum_{w\in X^N}ww^\ast)\in \tightid{K}{S_0}$.  Thus $K_0S_0/\tightid{K}{S_0}$ embeds in $\N{K}{G}{X}$ as the homogeneous component of degree zero.

Since $E(S_0)=E(S)$, we have that $\G_T(S_0)$ is the open subgroupoid of $\G_T(S)$ consisting of all germs of elements of $S_0$ (note that this subgroupoid is the kernel of the natural continuous cocycle $\G_T(S)\to \mathbb Z$ sending the germ $[ugv^\ast,w]$ to $|u|-|v|$).
Being an open subgroupoid of an effective groupoid, $\G_T(S_0)$ is effective.  It is also minimal.  For if $w\in X^{\omega}$ and $uX^{\omega}$ is a basic open set of the unit space $X^{\omega}$, then if $v$ is the prefix of $w$ of length $|u|$, and so $w=vz$, then $[uv^\ast,w]$ is an arrow from $w$ to $uz$, whence $uX^{\omega}$ intersects the orbit of $w$.  Thus every orbit is dense in $X^\omega$ and so $\G_T(S_0)$ is minimal.  It now follows from the results of~\cite{simplicity} that $K_0S_0/\tightid{K}{S_0}$ is simple if and only if $\sing{K}{S_0}=\tightid{K}{S_0}$.  By our previous observations, we deduce that $K_0S_0/\tightid{K}{S_0}$ is simple if and only if no element of $\sing K {S_0}\setminus \tightid K {S_0}$ belongs to $K_0S_0$, which holds if and only if no element of $\sing K {S}\setminus \tightid K {S}$ belongs to $K_0S_0$. Since $KG\leq K_0S_0$, Proposition~\ref{p:supportinG} then implies that $\N{K}{G}{X}$ is simple if and only if $K_0S_0/\tightid{K}{S_0}$, the homogeneous component of degree $0$,  is simple.
\end{Rem}

\subsection{The simplicity graph} Let $G$ be a self-similar group and $A \subseteq G$ an automaton. Our goal is to provide a computable criterion to check whether there is an element of $\sing K S\setminus \tightid K S$ supported on $A$. We then apply this to the nucleus of a contracting group.

If an $X$-generated monoid acts on the left of a set $V$, then the \emph{Schreier graph} of the action is the edge-labeled digraph with vertex set $V$ and directed edges of the form
$v \xrightarrow{\,\,x\,\,} xv$ for $x\in X$ and $v\in V$. Note that, for each $x\in X$, there is exactly one edge labeled $x$ leaving any vertex, and hence the Schreier graph is finite if both $V$ and $X$ are finite.  One can similarly define the Schreier graph of a right action of an $X$-generated monoid on $V$.  In this case, edges are of the form $v \xrightarrow{\,\,x\,\,} vx$.   A Schreier graph has no sinks.

We define a family of equivalence relations $V=\{\eq w: w \in X^\ast\}$ on $A$, recursively. Let $\eq \empty$ be the equality relation. For any $w \in X^\ast$, $x\in X$ and $g,h \in A$, put
$g \eq {xw} h$ whenever $g(x)=h(x)$ and $g|_x \eq w h|_x$.
Notice that the set $\{\eq w: w \in X^\ast\}$ is finite as $A$ is finite.

The set $V$ arises as the orbit of the equality relation under a certain left action of  $X^\ast$ on the set of equivalence relations on $A$.  For $x\in X$ and $\equiv$ an equivalence relation on $A$,  put $g \mathrel{(x \cdot{\equiv})} h$ whenever $g(x) = h(x)$ and $g|_x \equiv h|_x$ for $g,h\in A$.  The action is then extended to words recursively by having the empty word act identically and putting $(xw)\cdot {\equiv}$ equal to $x\cdot (w\cdot{\equiv})$ for $x\in X$ and $w\in X^\ast$.  We observe that $\eq w$ is precisely $w\cdot{\eq\empty}$, and $\eq\empty$ is the equality relation.

We define $\Sg_A$ to be the Schreier graph of the left action of $X^\ast$ on $V$ obtained by restriction of the action on equivalence relations.
For $u,w\in X^\ast$, there is a path labeled by $\rho(u)$ from $\equiv_w$ to $\equiv_{uw}$ where $\rho(u)$ denotes the reversal of the word $u$.

\begin{Prop}
\label{p:recursive.def}
For any $g,h \in A$, $w \in X^\ast$ and equivalence relation $\equiv$ on $A$, we have that $g\mathrel{(w\cdot {\equiv})} h$ if and only if $g(w)=h(w)$ and $g|_w\equiv h|_w$.   In particular,  we have $g \eq w h$ if and only if $g(w)=h(w)$ and $g|_{w}=h|_{w}$, that is, if and only if $gw=hw$ in $M=X^\ast G\leq S$.
\end{Prop}
\begin{proof}
We prove this by induction on $|w|$. If $|w|=0$, then $g(\varepsilon) = \varepsilon = h(\varepsilon)$ and $g|_\varepsilon=g$, $h|_{\varepsilon} =h$, and so there is nothing to prove.  Assume the proposition is true for $w$ of length $n$ and let $x\in X$.  Then $g|_{xw} = (g|_x)|_w$, $h|_{xw} = (h|_x)|_w$ and $g(xw) = g(x)g|_x(w)$, $h(xw) = h(x)h|_x(w)$.  Thus $g(xw)=h(xw)$ and $g|_{xw}\equiv h|_{xw}$  is equivalent to $g(x)=h(x)$, $g|_x(w)=h|_x(w)$ and   $(g|_x)|_w \equiv (h|_x)|_w$.  By induction, this is equivalent to $g(x)=h(x)$ and $g|_x\mathrel{(w\cdot {\equiv})} h|_x$, that is, $g\mathrel{(xw\cdot {\equiv})} h$.  The final statement follows by taking $\equiv$ to be the equality relation $\eq \varepsilon$ and recalling that ${\eq w}$ is $w\cdot{\eq\varepsilon}$.
\end{proof}

The strongly connected components of $\Sg_A$ are partially ordered by reachability: we put $\mathcal C\leq \mathcal C'$ if there is a path from $\mathcal C'$ to $\mathcal C$. We call the minimal elements \emph{minimal components}; these are precisely the strongly connected components that no edges leave. Denote the set of vertices contained in minimal components by $\Vmin$ and the set of essential vertices by $\Vess$.  Note that $\Vmin\subseteq \Vess$.  We wish to show that there is exactly one minimal component of $\Sg_A$; moreover, we show that $\Sg_A$ is synchronizing, that is, there are words $z$ such that any path  labeled by $z$  ends at the same vertex.

Given any word $u \in X^\ast$,
consider the set
\[V_{u}=\{{\eq {uw}}: w\in X^\ast\}.\]  It is the image of $u$ under the left action of $X^\ast$ on $V$ discussed above.

\begin{Prop}\label{p:reset.word}
We have that \[1=\min_{u\in X^\ast} |V_u|.\]  Moreover,  if $|V_u|=1$, then $V_{uz}=V_u=\{\equiv_u\}$  for all $z\in X^\ast$ and $|V_w|=1$ for all $w\in X^\ast uX^{\ast}$.
\end{Prop}
\begin{proof}
Note that, for all $u,z\in X^\ast$, we have ${\equiv_u}\subseteq {\equiv_{uz}}$, as $g\equiv_u h$ implies $gu=hu$ in $M=X^\ast G$, and hence $guz=huz$ in $M$, i.e., $g\equiv_{uz} h$ by Proposition~\ref{p:recursive.def}. In particular, $\equiv_u$ is contained in every equivalence relation in $V_u$.   Also, note that $V_{uz}\subseteq V_u$ by definition.  Therefore, if $|V_u|$ is minimal, then $V_u=V_{uz}$, and so ${\equiv_u}= {\equiv_{uz}}$, as the unique smallest equivalence relations in $V_u$ and $V_{uz}$ are $\equiv_u$ and $\equiv_{uz}$, respectively.  Since this is true for all $z\in X^\ast$, we deduce that $V_u=\{\equiv_u\}$ is a singleton.    If $V_u$ is a singleton, then since, for $z\in X^\ast$, the set $V_{zu}$ is the image of $V_u$ under the action of $z$ on equivalence relations in $V$, we deduce that $V_{zu}$ is a singleton, too. This proves the final assertion.
\end{proof}

As a corollary, we show that $\Sg_A$ has a unique minimal component.

\begin{Cor}\label{c:one.min.comp}
There is exactly one minimal component in $\Sg_A$.  An equivalence relation $\equiv$ belongs to the minimal component if and only if ${\equiv} = {\equiv_u}$ for some word $u\in X^\ast$ with $|V_u|=1$.
\end{Cor}
\begin{proof}
Let $u\in X^\ast$ with $|V_u|=1$, as per Proposition~\ref{p:reset.word}.  Then, for any vertex $\equiv_w$, we have that ${\equiv_{uw}}\in V_u = \{\equiv_u\}$, and so $\rho(u)$ labels a path from $\equiv_w$ to $\equiv_u$.  It follows that the strongly connected component of $\equiv_u$ is below every strongly connected component and hence is the unique minimum component.  The second statement follows because if $\equiv$ is some element of the minimum component, then ${\equiv}={\equiv_{zu}}$ for some $z\in X^\ast$ (with $\rho(z)$ labelling a path from $\equiv_u$ to $\equiv$), and so $|V_{zu}|=1$ by Proposition~\ref{p:reset.word}.
\end{proof}

For each relation ${\equiv} \in V$, we introduce the following system of linear equations over $\mathbb Z$ in variables $c_g$ with $g\in A$:
\[E_{\equiv}=\left\{\sum_{g \equiv h}c_g=0: h \in A\right\},\]
and let
\[E_{S}=\bigcup_{{\equiv} \in \Vmin} E_{\equiv}.\]
For brevity, we denote $E_{\eq w}$ by $E_{w}$.

Given any field $K$ and a finite homogeneous linear system of equations $E$ over $\mathbb Z$, denote the image of $E$ in the prime field of $K$ by $E_K$. Then $E_K$ has a solution in $K$ if and only if it has a solution in the prime field, and so the existence of solutions only depends on the characteristic of $K$.

If $E$ is any linear system of equations over $\mathbb Z$ in variables $c_g$ with $g\in A$, we say that $a =\sum_{g \in A}a_g g \in KG\leq K_0S$ satisfies $E_K$ if putting $c_g=a_g$ yields a solution of $E_K$.

For any equivalence relation $\equiv$ on $A$, there is a natural $K$-linear map $\pi_{\equiv}\colon KA \to K[A/{\equiv}]$, induced by the projection $A\to A/{\equiv}$, defined by $a=\sum_{g \in A}a_g g \mapsto \sum_{g \in A}a_g [g]_{\equiv}$ where $[g]_{\equiv}$ is the equivalence class of $g$. Then $a$ is in kernel of this map if and only if
$\sum_{g \equiv h}a_g=0$ for all $h \in A$, that is, if $a$ satisfies $E_{\equiv, K}$.

Notice that $A/{\equiv}_w$ is in bijection with $Aw$ via $[g]_{\equiv_w} \mapsto gw$ by Proposition~\ref{p:recursive.def}. Under this identification, $\pi_{\equiv_w}$ is just right multiplication by $w$, and so its kernel consists of those elements $a$ with $aw=0$. We have thus proved the following lemma:

\begin{Lem}
\label{lem:aw=0}
If $a=\sum_{g \in A}a_g g \in KA$, then $aw=0$ if and only if $a$ satisfies $E_{w,K}$.
\end{Lem}

\begin{Thm}
\label{t:singeqn}
An element $a=\sum_{g \in A}a_g g$ of $KA$ is singular in $K_0S$ if and only if it satisfies $E_{S,K}$.
\end{Thm}

\begin{Proof}
First assume $a$ satisfies $E_{S,K}$.
Let $u \in X^\ast$ be any word. We need to show that there exists $w \in X^\ast$ such that $auw=0$. By Lemma~\ref{lem:aw=0}, it suffices to find $w \in X^\ast$ with $\equiv_{uw}$ in $\Vmin$. Take any $w \in X^\ast$ with $\equiv_w$ in $\Vmin$. Then there is a path labeled by $\rho(u)$ from $\equiv_w$ to $\equiv_{uw}$, and so $\equiv_{uw}\ \in \Vmin$ as well.

For the converse, suppose that $a$ is singular.
By Corollary~\ref{c:one.min.comp}, to show that $a$ satisfies $E_{S,K}$ it suffices to show that $a$ satisfies $E_{u,K}$ whenever $|V_u|=1$.  Choose $w \in X^\ast$ such that $auw=0$; this exists as $a$ is singular.  Then  ${\equiv_{uw}}\in V_u=\{\equiv_u\}$, and so ${\equiv_{uw}}={\equiv_u}$.  Thus $E_{u,K}=E_{uw,K}$, and the latter is satisfied by $a$ by Lemma~\ref{lem:aw=0}.  This completes the proof.
\end{Proof}

It may happen that a singular element $a$ satisfies the equations of more vertices than those in $\Vmin$.  In order to get $a$ not to belong to $\tightid{K}{S}$ we have to make sure that it does not satisfy the equations of too many vertices.

\begin{Prop}
\label{p:butisittight}
An element $a=\sum_{g \in A}a_g g$ of $KA$ belongs to $\tightid K S$ if and only if it satisfies $E_{\equiv, K}$ for every essential vertex $\equiv$ of $\Sg_A$.
\end{Prop}
\begin{Proof}
If $a\in \tightid K S$, then $aX^k=0$ for some $k\geq 0$ by Proposition~\ref{p:tightideal}.  By Lemma~\ref{lem:aw=0}, we deduce that $a$ satisfies $E_{w,K}$ for all $w$ with $|w|\geq k$.   If $\equiv$ is essential, then there is a directed path of length $k$ from a vertex $\eq u$ to  $\equiv$  by Lemma~\ref{l:essential}.  Thus ${\equiv}={\eq {vu}}$ for some $v\in X^k$, and so $a$ satisfies $E_{\equiv, K}=E_{vu,K}$ as $|vu|\geq k$.  Conversely, suppose that $a$ satisfies $E_{\equiv, K}$ for every essential vertex $\equiv$ and let $n$ be the number of vertices of $\Sg_A$.  Then any vertex $\eq w$ with $|w|=n$ is reachable from the equality relation by a path of length $n$, and hence is essential by Lemma~\ref{l:essential}(5).  Thus $a$ satisfies $E_{w,K}$ by assumption, whence $aw=0$ by Lemma~\ref{lem:aw=0}. Therefore, $a\in \tightid K S$ by Proposition~\ref{p:tightideal}(4).
\end{Proof}

We say that a homogeneous system of linear equations $E'$ over a field $K$ is a consequence of a homogeneous system of linear equations $E$ if each solution of $E$ is also a solution of $E'$.  This means that augmenting the coefficient matrix of $E$ by the coefficient matrix of $E'$ does not change the rank.

\begin{Cor}
\label{c:find.sing}
Let $G$ be a self-similar group and let $A \subseteq G$ be an automaton. Then $KA \cap \sing K S \setminus \tightid K S$ is non-empty if and only if there is a vertex  ${\equiv}\in \Vess\setminus \Vmin$ such that $E_{\equiv,K}$ is not a consequence of $E_{S,K}$.
\end{Cor}
\begin{proof}
If $a\in KA\cap \sing K S\setminus \tightid K S$, then by Proposition~\ref{p:butisittight}, there is an essential vertex $\equiv$ of $\Sg_A$ such that $a$ does not satisfy $E_{\equiv,K}$.   But since $a\in \sing K S$, we have that $a$ satisfies $E_{S,K}$ by Theorem~\ref{t:singeqn}, and so ${\equiv}\notin \Vmin$ and  $E_{\equiv, K}$ is not a consequence of $E_{S,K}$.

Conversely, if there exists an essential vertex $\equiv$ as in the statement of the corollary, then there exists $a\in KA$ that satisfies $E_{S,K}$ but not $E_{\equiv, K}$.  Then $a\in KA\cap \sing K S\setminus \tightid K S$ by Theorem~\ref{t:singeqn} and Proposition~\ref{p:butisittight}.
\end{proof}

Corollary~\ref{c:find.sing} and Corollary~\ref{c:simplereducedtoKN} then lead to the following criterion for simplicity of the Nekrashevych algebra of a contracting group.

\begin{Thm}
\label{t:contract.crit}
Let $G$ be a contracting self-similar group over the alphabet $X$, $K$ a field and $A$ an automaton containing the nucleus of $G$.  Then $\N{K}{G}{X}$ is simple if and only if $E_{\equiv,K}$ is a consequence of $E_{S,K}$ for each  ${\equiv}\in \Vess\setminus \Vmin$.
\end{Thm}

The following fact about systems of equations over $\mathbb Z$ is well known, but we include here the argument for completeness.

\begin{Prop}\label{p:smith}
Let $A$ be an $m\times n$ integer matrix.  Then the $\mathbb Q$-rank of $A$ agrees with the $\mathbb Z_p$-rank of $A$ for all but finitely many primes $p$ and one can compute this finite set of primes and the corresponding rank in polynomial time.
\end{Prop}
\begin{Proof}
Any integer matrix can be brought using elementary row and column operations over $\mathbb Z$ into Smith normal form in polynomial time.  Smith normal form is diagonal with non-zero diagonal entries $d_1\mid d_2\mid\cdots \mid d_r$ and $r$ is the $\mathbb Q$-rank of $A$.  If $p$ does not divide $d_r$, then the $\mathbb Z_p$-rank of $A$ is $r$.  Otherwise, the $\mathbb Z_p$-rank of $A$ is $i-1$ where $i$ is minimum with $p\mid d_i$.
\end{Proof}

The following algorithm lets one determine if an automaton supports a singular element not belonging to $\tightid K S$.

\begin{Thm}
\label{t:algorithmA}
Let $A$ be an automaton given by a state diagram.  Then for any self-similar group $G$ containing $A$ with associated inverse semigroup $S$,  either $KA \cap \sing K S \setminus \tightid K S$ is non-empty for every field $K$, or $KA\cap \sing K S \setminus \tightid K S=\emptyset$ for all fields $K$ except for those with characteristic belonging to a finite set of primes. Moreover, there is an algorithm that on input $A$, outputs which of the two cases holds and in the second case outputs the finite set of characteristics of fields $K$ such that $KA \cap \sing K S \setminus \tightid K S$ is non-empty.
\end{Thm}
\begin{Proof}
First note that
the proof of~\cite[Corollary~5.14]{simplicity} shows
that if $a\in \sing {\mathbb Q} S\setminus \tightid {\mathbb Q} S$, then for each prime $p$, there is a multiple $da\in \mathbb Z_0S$ with $d\in \mathbb Q\setminus \{0\}$ such that $da$ maps to an element of $\sing {\mathbb Z_p} S\setminus \tightid {\mathbb Z_p} S$.

The state diagram of the automaton of $A$ lets us compute the (left) action of elements of $A$ on $X$ and the (right) action of $X^\ast$ on $A$. The algorithm then is as follows:

Step 1: Build the graph $\Sg_A$ recursively starting from the vertex of the equality relation. For the recursion, for each vertex $\equiv$ that we have constructed and each letter $x \in X$, we compute $x \cdot{\equiv}$; if this is an existing vertex, we add the respective edge ${\equiv} \xrightarrow{\,\,x\,\,}\ x \cdot{\equiv}$, otherwise we add both the new vertex and the edge. This procedure eventually stops as $\Sg_A$ is finite.

Step 2: Compute the vertices $\Vmin$ of the minimal component
 and from these the set of linear equations $E_{S}$.

Step 3: List all the vertices of $\Vess\setminus \Vmin$: $\equiv_1, \ldots, \equiv_m$.  Let $B$ be the coefficient matrix for the system $E_{S}$ and let $B'$ be the coefficient matrix for the system $\bigcup_{i=1}^m E_{\equiv_i}\cup E_{S}$ (over $\mathbb Z$); these are integer matrices. It follows from Corollary~\ref{c:find.sing} that $KA\cap \sing K S \setminus \tightid K S$ is non-empty if and only if the rank of $B'$ over $K$ is greater than the rank of $B$ over $K$ (where we view integer matrices over any field by projecting into the prime field).

If the $\mathbb Q$-rank of $B'$ is larger than the $\mathbb Q$-rank of $B$, then, for any field $K$ of characteristic zero, we have  $KA\cap \sing K S \setminus \tightid K S$ is non-empty  and also, by the observation at the beginning of the proof, we have $\mathbb Z_pA\cap \sing {\mathbb Z_p} S\setminus \tightid {\mathbb Z_p} S$ for every prime $p$ and so $KA\cap \sing K S\setminus \tightid K S$ is non-empty for every field $K$.

On the other hand, if $B$ and $B'$ have the same rank over $\mathbb Q$, then by Proposition~\ref{p:smith}, there is a finite set of primes where the ranks of $B$ and $B'$ differ from their $\mathbb Q$-ranks, and, moreover, we can compute this finite set of primes and their ranks over each of these primes.  Thus we can output the finite set of primes $p$ for which the $\mathbb Z_p$-rank of $B'$ is bigger than that of $B$.
\end{Proof}

In practice the equations $E_{S,K}$ often have no non-trivial solutions and so one should check that first.  Note that both Steps~2 and~3 can be done in polynomial time in the size of $\Sg_A$; but Step~1 is more complicated, as there can be many equivalence relations on $A$, and so we currently have no non-trivial time bound on Step~1 or size bound on $\Sg_A$.

\begin{Rem}
The algorithm is constructive in the sense that if $KA \cap \sing K S \setminus \tightid K S$ is non-empty, we can also find an element by finding a solution of $E_{S, K}$ which does not solve some $E_{S,\equiv_i}$.
\end{Rem}

We now arrive at one of our main results:

\begin{Thm}
\label{t:contracting}
Let $G$ be a contracting self-similar group acting over a finite alphabet $X$ given by the state diagram of a finite automaton $A$ generating $G$. Then then the simplicity of the Nekrashevych algebra $\N{K}{G}{X}$ depends only on the characteristic of $K$: it is either non-simple over all characteristics, or it is simple over all fields but those of finitely many positive characteristics. Furthermore, there is an algorithm which decides whether we are in the first case or, in the second case, outputs the finite set of characteristics  such that $\N{K}{G}{X}$ is not simple.
\end{Thm}
\begin{proof}
There is an well-known algorithm (described in Section~\ref{s:ssg}) that computes the nucleus of $G$ from $A$ whenever $G$ is contracting.  The result then follows from Corollary~\ref{c:simplereducedtoKN} and Theorem~\ref{t:algorithmA}.
\end{proof}

We do not know of an example of  a (non-contracting) self-similar group over a finite alphabet whose Nekrashevych algebra is simple in characteristic $0$ but is non-simple over fields of infinitely many positive characteristics.

A natural question arises as to whether the Nekrashevych $C^\ast$-algebra of a self-replicating, contracting group $G$ over a finite alphabet is simple if and only if the Nekrashevych algebra of $G$ over the complex numbers is simple (simplicity of the complex Nekrashevych algebra is known to be a necessary condition~\cite{nonhausdorffsimple}).  Note that the tight groupoid of the inverse semigroup $S$ is amenable in this case and so its reduced $C^\ast$-algebra is the Nekrashevych $C^\ast$-algebra.  In~\cite{nonhausdorffsimple} it was shown that the Nekrashevych $C^\ast$-algebra of the Grigorchuk group is simple.  More generally, Yoshida~\cite{multispinalCstar} has shown the question has a positive answer for the class of multispinal groups defined in Section~\ref{s:multispinal}.

To close this section, we characterize the Hausdorff property of the groupoid associated to a contracting group $G$ in terms of the simplicity graph. In particular, if the groupoid is Hausdorff, then all essential vertices of the simplicity graph belong to the minimal component, and thus Theorem~\ref{t:contract.crit} recovers the simplicity of the Nekrashevych algebra. The converse of this statement does not hold (we shall see an example later coming from Gupta-Sidki groups).

\begin{Prop}\label{p:NHd}
Let $G$ be a contracting self-similar group and let $A$ be a finite automaton containing the nucleus $N$ of $G$.   Then the groupoid associated to $G$ is Hausdorff if and only if, there exists $n\geq 0$, such that for all $w\in X^\ast$ of length $n$,  we have $|V_w|=1$ in $\Sg_A$.  In particular, if the groupoid is Hausdorff, then $\Vmin=\Vess$.
\end{Prop}
\begin{proof}
Suppose first that $|V_w|=1$ whenever $|w|=n$.  By the proof of Proposition~\ref{p:Hd}, it suffices to show that, for each $g\in N$, there is a finite set $F_g$ such that $F_gX^\ast$ is the set of words strongly fixed by $g$.  Let $F_g$ be the set of words of length at most $n$ strongly fixed by $g$. It is clear that any word in $F_gX^\ast$ is strongly fixed by $g$, so we only need to show the converse. If $u$ of length greater than $n$ is strongly fixed by $g$, write $u=wx$ with $|w|=n$.  By assumption,  $|V_w|=1$ and so ${\equiv_w}={\equiv_u}$ by Proposition~\ref{p:reset.word}.  By Proposition~\ref{p:recursive.def}, $gu=u$ in $M=X^\ast G$ implies that $g\equiv_u 1$, and hence $g\equiv_w 1$. Therefore,  $gw=w$ in $M$ by another application of Proposition~\ref{p:recursive.def}.  Thus $w\in F_g$ and $u\in F_gX^\ast$.  Since $F_g$ is finite, we deduce that the groupoid associated to $G$ is Hausdorff.

Next suppose that the associated groupoid is Hausdorff and that, by way of contradiction, we can find words $w_1,w_2,\cdots$ in $X^\ast$ with $|w_1|<|w_2|<\cdots$ and $|V_{w_n}|>1$ for all $n\geq 1$.  For each $n$, choose $z_n$ with ${\equiv_{w_nz_n}}\neq {\equiv_{w_n}}$.  Since ${\equiv_{w_n}}\subseteq {\equiv_{w_nz_n}}$ by the proof of Proposition~\ref{p:reset.word}, we can find $g_n,h_n\in N$ with $g_n\equiv_{w_nz_n} h_n$, but $g_n\not\equiv_{w_n} h_n$.  Moreover, we may assume that $z_n$ is the shortest word such that $g_n\equiv_{w_nz_n} h_n$.  Since $A$ is finite, by passing to  subsequence and reindexing we may assume that there are $g,h\in A$ with $g=g_n$ and $h=h_n$ for all $n$.  Then $gw_nz_n=hw_nz_n$  in $M=X^\ast G$, for all $n\geq 1$, but $gw_nz'\neq hw_nz'$ for any proper prefix $z'$ of $z_n$ by Proposition~\ref{p:recursive.def}. Thus $h^{-1}g$ strongly fixes each $w_nz_n$, but no proper prefix of $w_nz_n$.  But since the groupoid is Hausdorff, there is by~\cite[Theorem~12.2]{ExPadKatsura} a finite set $F_{h^{-1}g}$ with $F_{h^{-1}g}X^\ast$ the set of words strongly fixed by $h^{-1}g$.  We must have $w_nz_n\in F_{h^{-1}g}$ for each $n$ (as no proper prefix of $w_nz_n$ is strongly fixed), but since $|w_n|\to \infty$, this is a contradiction.

For the final statement, assume that the associated groupoid is Hausdorff and that $n\geq 0$ is such that $|V_w|=1$ for all $w$ of length $n$.  Let $\equiv$ be an essential state.  By Lemma~\ref{l:essential} there is a directed path of length $n$ ending at $\equiv$ in $\Sg_A$. Since all vertices are reachable from equality, we may find a word $w$ with $|w|\geq n$ and ${\equiv}={\equiv_w}$.  Then $w=uv$ with $|u|=n$ and hence $|V_w|=1$ by Proposition~\ref{p:reset.word}.  Thus $\equiv$ is minimal by Corollary~\ref{c:one.min.comp}.
\end{proof}

We remark that if $\Sg_A$ has $m$ vertices, one can take $n=m^m$ in Proposition~\ref{p:NHd} by a standard argument of finite semigroup theory.

Note that Proposition~\ref{p:NHd} and Theorem~\ref{t:contract.crit} give an alternative proof that the Nekrashevych algebra of a contracting group with Hausdorff groupoid is simple over any field.

\section{Elementary examples}\label{s:ex}
In this section we first consider the simplicity of the Nekrashevych algebra of the Basilica group.   The corresponding groupoid is Hausdorff (as was pointed out to us by Nekrashevych); nonetheless, we compute its  simplicity graph as an illustration of Proposition~\ref{p:NHd}. We also present a general construction of self-similar groups with a non-simple Nekrashevych algebra over every field.

\subsection{The Basilica group}

The \emph{Basilica group} $B$ is the iterated monodromy group of the polynomial $z^2-1$.  It  is a self-similar group acting on $X=\{0,1\}$ and was first studied by Grigorchuk and \.{Z}uk~\cite{BasilicaGZ}.  It is generated by the automaton with states the identity $1_B$ and $a,c$ given by $a(0w) =0w$, $a(1w) = 1c(w)$, $c(0w) = 1w$, $c(1w) = 0a(w)$ for $w\in \{0,1\}^\ast$.  The nucleus $N$ is well known to consist of the elements  $\{1_B,a,b,c,d,e,f\}$, where $b=a^{-1}$, $d=c^{-1}$, $e=ca^{-1}$ and $f=ac^{-1}$.  Note that $1_B,a,b$ act on $X$ identically and $c,d,e,f$ by the transposition.
The state diagram of the nucleus is in Figure~\ref{f:nucleus.basilica}.
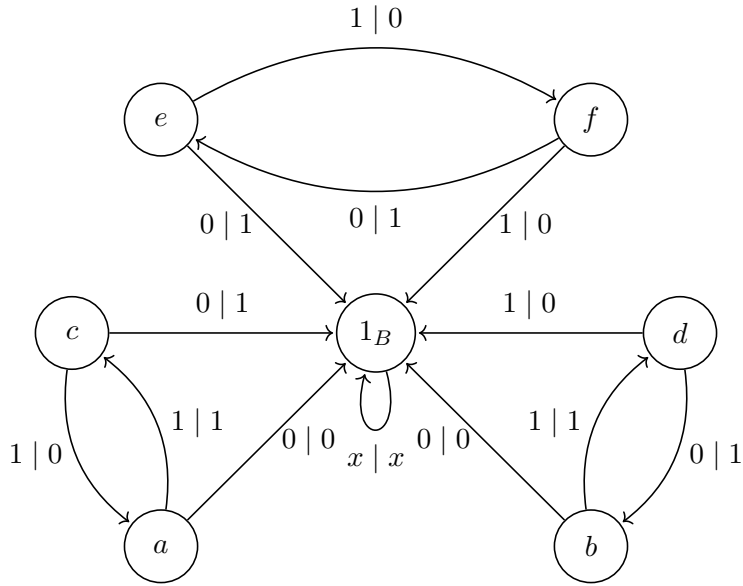
\begin{figure}[htbp]
\begin{center}
\begin{tikzpicture}[->,shorten >=1pt,%
auto,node distance=4cm,semithick,
inner sep=5pt,bend angle=30]
\tikzset{every loop/.style={min distance=10mm,looseness=10}}
%\tikzset{every state/.style={minimum size=30pt}}.
\node[state] (I) {$1_B$};
\node[state] (E) [above left of=I] {$e$};
\node[state] (F) [above right of=I] {$f$};
\node[state] (C) [left of=I]{$c$};
\node[state] (A) [below left of=I]{$a$};
\node[state] (D) [right of=I]{$d$};
\node[state] (B) [below right of=I]{$b$};
%\tikzstyle{every node}=[font=\footnotesize]
\path   (I) edge [loop below] node [below] {$x\mid x$} (I)
        (A) edge [bend right] node [right] {$1\mid 1$} (C)
        (A) edge [right]      node [right] {$0\mid 0$} (I)
        (C) edge [bend right] node [left] {$1\mid 0$} (A)
        (C) edge [right]      node [above] {$0\mid 1$} (I)
        (B) edge [bend left] node [left] {$1\mid 1$} (D)
        (B) edge [left]      node [left] {$0\mid 0$} (I)
        (D) edge [bend left] node [right] {$0\mid 1$} (B)
        (D) edge [left]      node [above] {$1\mid 0$} (I)
        (E) edge [bend left] node [above] {$1\mid 0$} (F)
        (E) edge [left]      node [left] {$0\mid 1$} (I)
        (F) edge [bend left] node [below] {$0\mid 1$} (E)
        (F) edge [left]      node [right] {$1\mid 0$} (I);
\end{tikzpicture}
\end{center}
\caption{State diagram for the nucleus of the Basilica group~\label{f:nucleus.basilica}}
\end{figure}

It can easily be seen from Proposition~\ref{p:Hd} that the associated groupoid is Hausdorff: the graph $\mathcal H$ from that proposition is induced by the vertices $a, b$ and $1_B$, and contains no cycles other than the loops around $1_B$. Thus the Nekrashevych algebra is simple over any field  and   $\Vmin=\Vess$ by Proposition~\ref{p:NHd}.  Figure~\ref{f:basilica} displays the edge-labeled graph $\Sg_N$ (which can be readily verified by direct computation).

\begin{figure}[htbp]
\begin{center}
\begin{tikzpicture}[scale=1.5,semithick]
\tikzstyle{vertex}=[circle, draw, minimum size=32pt]
\tikzstyle{edge}=[->] %this makes the arrowhead nice

\foreach \name/\pos/\label in {{=/(-4,0)/$=$}, {0/(-2,1)/$\eq 0$}, {1/(-2,-1)/$\eq 1$}, {00/(0,2)/$\eq{00}$}, {10/(0,0)/$\eq{10}$}, {100/(2,1)/$\eq{100}$}, {110/(2,-1)/$\eq{110}$}}
\node[vertex] (\name) at \pos {\label}; %states

\foreach \from/\to/\label in {{=/0/0},{0/00/0}, {0/10/1}, {00/100/1}, {10/110/1}}
\path[edge] (\from) edge[bend left=15] node[above] {$\label$} (\to); %edges curving up from left to right

\foreach \from/\to/\label in {{10/0/0}, {110/1/1}}
\path[edge] (\from) edge[bend left=15] node[below] {$\label$} (\to);
%edges curving down from right to left

%sporadic:
\path[edge] (=) edge[bend right=15] node[below] {$1$} (1);
\path[edge] (100) edge[bend left=15] node[right] {$1$} (110);
\path[edge] (100) edge[bend right=10] node[above] {$0$} (0);
\path[edge] (110) edge[bend left=30] node[below] {$0$} (0);
\path[edge] (1) edge[bend left=15] node[left] {$0$} (0);

%loops:
\path[edge] (1) edge[in=250, out=290, looseness=10] node[below] {$1$} (1);
\path[edge] (00) edge[in=70, out=110, looseness=10] node[above] {$0$} (00);

\end{tikzpicture}
\end{center}
\caption{The graph $\Sg_N$ for the Basilica group\label{f:basilica}}
\end{figure}
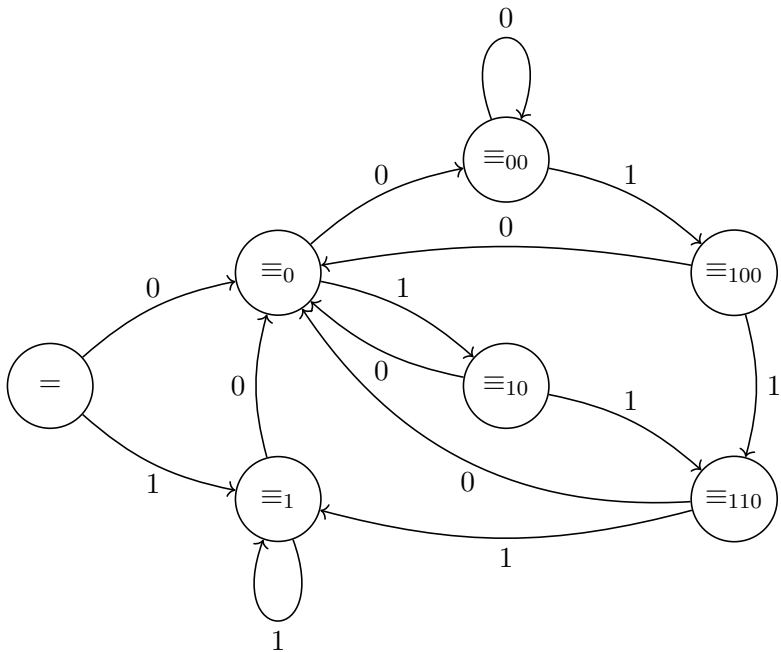

Notice that $\Vmin=V\setminus\{=\}$ and ${=}$ is inessential.    One can directly verify that $|V_w|=1$ for all words $w$ of length $3$, providing another proof that the groupoid associated to the Basilica group is Hausdorff via Proposition~\ref{p:NHd}.  Thus $\N{K}{B}{\{0,1\}}$ is simple over any field $K$ by Theorem~\ref{t:contract.crit} (or because its groupoid is Hausdorff).

\begin{Thm}
The Basilica group has a simple Nekrashevych algebra over every field.
\end{Thm}

\subsection{Some non-simple Nekrashevych algebras}
The first example of a non-simple Nekrashevych algebra over some field was given in~\cite{Nekrashevychgpd,nonhausdorffsimple}: the Nekrashevych algebra of the Grigorchuk group~\cite{grigroup} is not simple over fields of characteristic $2$, but is simple over fields of any other characteristic.  Nekrashevych showed (unpublished) that the Nekrashevych algebra of the Grigorchuk-Erschler group~\cite{Grigdegree,Erschler} is not simple over any field; more on this later.

Here is a straightforward method of constructing examples of self-similar groups with non-simple Nekrashevych algebras over any field.

If $G$ is a self-similar group over a finite alphabet $X$, define the \emph{trivial one-step inflation} of $G$ to be the self-similar action of $G$ on $X\coprod \{z\}$ obtained by extending the action on $X$ by putting $g(z)=z$ and $g|_z=g$ for any $g\in G$.

\begin{Prop}\label{p:trivial.inflate}
Let $G$ have a self-similar action over the finite alphabet $X$ and suppose that $KG\cap \tightid K S\neq 0$, where $S$ is the associated inverse semigroup.  Putting $Y=X\coprod \{z\}$ with the trivial one-step inflation action, $\N{K}{G}{Y}$ is not simple.
\end{Prop}
\begin{proof}
Let $0\neq a\in KG\cap \tightid K S$.  Then $aX^N=0$ for some $N\geq 0$ by Proposition~\ref{p:tightideal}.  By construction of the action $gz=zg$ for all $g\in G$ and hence $az=za$.  Thus $auX^N=0$ for any $u\in Y^\ast$.  We conclude that $a$ is singular.  But $az^n=z^na\neq 0$ for all $n\geq 0$ and so $a$ is not in the tight ideal over the larger alphabet $Y$ by Proposition~\ref{p:tightideal}(4).  Thus $\N{K}{G}{Y}$ is not simple.
\end{proof}

The following construction is inspired by two of the constructions in~
\cite[Section~6]{simplicity}.

\begin{Cor}
\label{c:direct.new}
Let $H_1,H_2$ be non-trivial self-similar groups over the respective alphabets $X_1,X_2$.  Define a self-similar action of $G=H_1 \times H_2$ on $X_1 \coprod X_2$ by $(h_1,h_2)(x) = h_i(x)$ if $x\in X_i$, for $i=1,2$, and
 \[(h_1,h_2)|_x = \begin{cases} (h_1|_x,1), & \text{if}\ x\in X_1\\ (1,h_2|_x), & \text{if}\ x\in X_2\end{cases}.\]  If $Y= X_1\coprod X_2\coprod \{z\}$ with the trivial one-step inflation action, then $\N{K}{G}{Y}$ is not simple for any field $K$.
\end{Cor}
\begin{proof}
By Proposition~\ref{p:trivial.inflate}, it suffices to show that if $S$ is the inverse semigroup associated to the action of $G$ on $X=X_1\coprod X_2$, then $KG\cap \tightid K S\neq 0$.  Let $a=((1,1)-(h_1,1))((1,1)-(1,h_2))= ((1,1)-(1,h_2))((1,1)-(h_1,1))\in KG$ where $1\neq h_1\in H_1$ and $1\neq h_2\in H_2$.  Then $((1,1)-(h_1,1))x=0$ for all $x\in X_2$ and $((1,1)-(1,h_2))x=0$ for all $x\in X_1$ by construction.  Thus $aX=0$, and so $0\neq a\in KG\cap \tightid K S$, as required.
\end{proof}

For the above construction, we can just take $H_1=H_2=\mathbb Z_2$ acting over $\{0,1\}$ by acting on the first letter only.  Note that the tight groupoids of the inverse semigroups associated to the groups in Corollary~\ref{c:direct.new} are minimal and effective, but have non-simple algebras over every field.

\section{Multispinal groups}\label{s:multispinal}

In this section, we consider a construction of contracting self-similar groups that generalizes self-replicating spinal automaton groups~\cite{spinalgroups} and, in particular, generalizes a construction of \v{S}uni\'{c}~\cite{sunicgroups}, which produces a natural family of self-similar groups containing the  Grigorchuk group~\cite{grigroup}.  We call self-similar groups obtained in this fashion \emph{multispinal groups}.  Multispinal groups also include the Gupta-Sidki $p$-groups~\cite{GuptaSidki}, \textsf{GGS}-groups~\cite{BGSbranch} and multi-edge spinal groups~\cite{multiedgespinal}.

\subsection{Construction}

If $G,H$ are groups, recall that $\mathrm{Aut}(G)$ acts on the right of $\mathrm{Hom}(G,H)$ via precomposition.
Let $X$ be a finite set.  A group $H$ acts \emph{freely} on $X$ if the stabilizer of each point is trivial; in particular, a free action is faithful.    The data needed to build a multispinal group over the alphabet $X$ are a finite group $G$, a finite group $H$ with a free left action on $X$ and a mapping $\Phi\colon X\to \mathrm{Aut}(G)\cup \mathrm{Hom}(G,H)$ meeting the following requirements:
\begin{enumerate}
  \item [(S1)] $\Phi(X)\cap \mathrm{Aut}(G)\neq \emptyset$;
  \item [(S2)] $\Phi(X)\cap \mathrm{Hom}(G,H)\neq \emptyset$;
  \item [(S3)] If $A=\langle \Phi(X)\cap \mathrm{Aut}(G)\rangle$ and $B=(\Phi(X)\cap \mathrm{Hom}(G,H))\cdot A$ (via the right action of $\mathrm{Aut}(G)$ on $\mathrm{Hom}(G,H)$), then $\bigcap_{\lambda\in B}\ker \lambda=\{1\}$.
\end{enumerate}

We define an automaton $\mathcal{A}$ with state set $(G\coprod H)/{\sim}$ where $\sim$ is the equivalence relation identifying the identities $1_G$ of $G$ and $1_H$ of $H$ into a class that we denote by $1$.  The action and sections are defined by the followings rules. If $h\in H$, then it acts on $X$ via the given action of $H$ on $X$ and $h|_x=1$ for all $x\in X$.  If $g\in G$, we put
\begin{gather*}
g(x) = x,\quad g|_x = \Phi(x)(g),\quad  \forall x\in X.
\end{gather*}
Notice that if $g=1$, then this gives the identity map on $X^\ast$, and so there is no ambiguity in the definition of the state $1$.
When $\Phi(x) \in \mathrm{Aut}(G)$, we have $g|_x \in G$, whereas if $\Phi(x) \in \mathrm{Hom}(G,H)$, then $g|_x \in H$.
Let $\mathfrak G$ be the automaton group generated by the states of $\mathcal{A}$.

For example, the \v{S}uni\'{c} groups~\cite{sunicgroups} are the special case where, for a prime $p$,  $X=\mathbb Z_p$, $G=\mathbb Z_p^n$, $H=\mathbb Z_p$ acting on the left of itself and $\Phi(X)$ contains exactly one automorphism  of $G$ and one non-trivial homomorphism from $G$ to $H$
(with suitable restrictions to make (S3) hold).  In particular, the Grigorchuk group~\cite{grigroup} is the multispinal group with $G=\mathbb Z_2\times \mathbb Z_2$, $H=\mathbb Z_2$ and
\[\Phi(0)(x,y) = y,\quad \Phi(1)=\begin{pmatrix} 0 & 1\\ 1& 1\end{pmatrix}.\]

\begin{Prop}\label{p:contains.nuc}
The multispinal group $\mathfrak G$ is contracting with nucleus contained in $\mathcal{A}$.  The states $G$ and $H$ generate isomorphic copies of $G$ and $H$ in $\mathfrak G$, respectively.  The nucleus is $N= G\cup \bigcup_{x\in \Phi^{-1}(\mathrm{Hom}(G,H))}\Phi(x)(G)$.
\end{Prop}
\begin{Proof}
The states from $H$ generate a copy of $H$, as $H$ acts faithfully on $X$.  Since $\Phi(x)$ is a homomorphism, for each $x\in X$, we have that $(gg')|_x=\Phi(x)(gg') = \Phi(x)(g)\Phi(x)(g')= g|_x g'|_x$ for all $g,g'\in G$.  Thus $(gg')(xw) = x(gg')|_x(w) = x(g|_xg'|_x)(w)=xg|_x(g'|_x(w)) = g(xg'|_x(w)) = g(g'(xw))$, and so the assignment of $g\in G$ to its corresponding state is a homomorphism.  This assignment is injective because if $1\neq g\in G$, then we can find $\lambda\in B$ with $\lambda(g)\neq 1$ by (S3).  But, by construction, there is $w\in X^*$ with $g|_w=\lambda(g)\neq 1$.  Thus $G$ acts faithfully on $X^*$.  It now follows that $\mathcal{A}$ is closed under inversion with the inverse of the state $k\in G\cup H$ being given by $k^{-1}$.

We next verify that $N$ is the nucleus.   We begin by observing that $N=G|_X$:  for if $\Phi(x)\in \mathrm{Aut}(G)$, then, for any $g\in G$, $g|_x\in G$ and $g= \Phi(x)(g')=g'|_x$ for some $g'\in G$, and so $G|_{\Phi^{-1}(\mathrm{Aut}(G))}=G$. Therefore,
\[G|_X=G|_{\Phi^{-1}(\mathrm{Aut}(G))} \cup G|_{\Phi^{-1}(\mathrm{Hom}(G,H))}=G \cup \bigcup_{x\in \Phi^{-1}(\mathrm{Hom}(G,H))}\Phi(x)(G)=N.\]
As $H|_X=\{1\}$, it follows now that $N|_X=N$.      By~\cite[Lemma~2.11.2]{selfsimilar} and the discussion thereafter, to show that $N$ is the nucleus, it suffices to show that $N=N^{-1}$ and $N^2|_X\subseteq N$.   Note that $N$ is closed under inversion by definition. Since $G^2=G$, $H^2=H$, $G|_X=N$ and $H|_X=\{1\}$, it remains to show that if $g \in G$, $h \in H$, then $gh$ and $hg$ have sections in $N$. Indeed, for $x\in X$, we have that $(gh)|_x=g|_{h(x)}h|_x=g|_{h(x)} \in N$, by definition, and $(hg)|_x=h|_{g(x)}g|_x=g|_x \in N$ again.   This completes the proof.
\end{Proof}

In most, but not all, examples that we consider $\mathcal A$ will be the nucleus.  In any event, since $\mathcal A$ contains the nucleus, Theorem~\ref{t:contract.crit}  will allows us to use $\Sg_\mathcal A$ to determine simplicity of the Nekrashevych algebra of a multispinal group.

We remark that if $H$ is transitive on $X$ and $H=\langle \bigcup_{x\in \Phi^{-1}(\mathrm{Hom}(G,H))}\Phi(x)(G)\rangle$, then $\mathfrak G$ will be self-replicating (and hence infinite) and its associated ample groupoid will be amenable by~\cite{Nekcstar}.  If $|\Phi(X)\cap \mathrm{Aut}(G)|=1$, then $\mathfrak G$ is a spinal group~\cite{spinalgroups} and all self-replicating spinal automaton groups are of this form.  In particular, when $\mathfrak G$ is spinal, $\mathcal A$ is a bounded automaton and so $\mathfrak G$ is amenable by~\cite{boundedaut}.  Multispinal groups for which $\Phi(X)$ contains exactly one automorphism and one non-trivial homomorphism are called \textsf{G}-groups~\cite{BGSbranch}.

\subsection{Representation theory}
Representation theory will play a key role in our analysis of the simplicity of Nekrashevych algebras of multispinal groups.  The reader is referred to~\cite[Chapter~3.8]{LamBook} as a reference on the representation theory of finite groups.
Let $G$ be a finite group and $K$ an algebraically closed field whose characteristic does not divide $|G|$.  Then by Maschke's theorem $KG$ is semisimple~\cite[Theorem~6.1]{LamBook}.  Let $\widehat{G}$ be the set of equivalence classes of irreducible representations of $G$ over $K$. Any representation $\rho\colon G\to M_n(K)$ can be extended to a $K$-algebra homomorphism $\rho\colon KG\to M_n(K)$ in the usual way. The character $\chi_{\rho}\colon G\to K$ of $\rho$ is the mapping taking $g$ to the trace of $\rho(g)$.  The character determines the representation up to equivalence.

 We shall fix a representative of each class from $\widehat{G}$ and identify the class with its representative.  Then Wedderburn's theorem~\cite{LamBook} yields an isomorphism
\begin{gather*}
\Psi_G\colon KG\to \prod_{\rho\in \widehat{G}} M_{n_{\rho}}(K)\\  a\mapsto (\rho(a))_{\rho\in \widehat{G}}
\end{gather*}
where $n_{\rho}$ is the degree of the representation $\rho$.  Notice that $G$ is abelian if and only if $KG$ is commutative, if and only if $n_{\rho}=1$ for all $\rho\in \widehat{G}$.

\begin{Rem}\label{r:explicit}
If $\rho$ is an irreducible representation of $G$ and  \[e_{\rho} = \frac{n_{\rho}}{|G|}\sum_{g\in G}\chi_{\rho}(g^{-1})g,\]  then it is well known~\cite[Proposition~8.15]{LamBook} that $\rho(e_{\rho}) = I_{n_{\rho}}$ and $\psi(e_{\rho})=0$ for $\psi\in \widehat{G}\setminus \{\rho\}$.  Hence $\Psi_G(e_{\rho})$ is the identity of the Wedderburn component $M_{n_\rho}(K)$.
\end{Rem}

Let $N\lhd G$ be a normal subgroup and $\pi_N\colon G\to G/N$ the quotient map.  If $\rho\in \widehat{G/N}$, then $\rho\circ\pi_N\in \widehat{G}$ and so we can view $\widehat{G/N}$ as embedded in $\widehat{G}$ as those irreducible representations whose kernel contains $N$.  Then the following diagram commutes by definition:

\begin{equation*}\label{fourier}
\begin{tikzcd}
KG\ar{r}{\Psi_G}\ar{d}[swap]{\pi_N} & \displaystyle{\prod_{\rho\in \widehat{G}} M_{n_{\rho}}(K)\ar{d}{\nu_N}} \\
K[G/N]\ar{r}{\Psi_{G/N}} & \displaystyle{\prod_{\rho\in \widehat{G/N}}}M_{n_{\rho}}(K)
\end{tikzcd}
\end{equation*}
where $\nu_N$ is the projection and $\Psi_G$, $\Psi_{G/N}$ are the Wedderburn isomorphisms. Therefore, $\ker \pi_N = \Psi_G^{-1}\left(\prod_{\rho\notin \widehat{G/N}}M_{n_{\rho}}(K)\right)$.

\begin{Prop}
\label{p:separate.points}
Let $G$ be a finite group and $K$ an algebraically closed field whose characteristic does not divide $|G|$.  Let $B$ be a collection of normal subgroups of $G$ and let $\pi_N\colon KG\to K[G/N]$ be the projection for $N\in B$.  Then the following are equivalent:
\begin{enumerate}
\item $\bigcap_{N\in B} \ker \pi_N=0$;
\item for each irreducible representation $\rho\in \widehat{G}$, there exists $N\in B$ with $N\leq \ker \rho$.
\end{enumerate}
\end{Prop}
\begin{proof}
It follows from the above discussion that \[\bigcap_{N\in B}\ker \pi_N = \bigcap_{N\in B}\Psi_G^{-1}\left(\prod_{\rho\notin \widehat{G/N}}M_{n_{\rho}}(K)\right)=\Psi_G^{-1}\left(\prod_{\rho\in \widehat{G}\setminus \bigcup_{N\in B}\widehat{G/N}}M_{n_{\rho}}(K)\right).\]  This will be $0$ if and only if $\widehat{G}=\bigcup_{N\in B}\widehat{G/N}$, i.e., (2) holds.
\end{proof}

\subsection{Simplicity for Nekrashevych algebras of multispinal groups}
The main theorem of this subsection describes simplicity of Nekrashevych algebras of multispinal groups. We will then use the theorem to give examples of finitely generated, infinite $p$-groups of intermediate growth, which have simple Nekrashevych algebras over all fields of characteristic different than $p$, but non-simple algebras in characteristic $p$.  When $p=2$, this includes the Grigorchuk group~\cite{grigroup}, which was first handled in~\cite{nonhausdorffsimple,Nekrashevychgpd}. We also show that, for any finite set $\mathcal P$ of primes, there is a multispinal group whose Nekrashevych algebra is simple over precisely those fields whose characteristic does not belong to $\mathcal P$.  Since contracting groups with simple Nekrashevych algebras in characteristic $0$ can only fail to have simple algebras over finitely many prime characteristics by Theorem~\ref{t:contracting}, this shows that one cannot constrain these characteristics in any further way.

We begin with a lemma describing $\Sg_{\mathcal A}$ for a multispinal group.

\begin{Lem}\label{l:gen.sunic.graph}
Let $G$ and $H$ be finite groups with $H$ acting freely on $X$ and $\Phi\colon X\to \mathrm{Aut}(G)\cup \mathrm{Hom}(G,H)$ be  the data defining a multispinal group.  Let $A,B$ be as in \textrm{(S3)} and $\mathcal A=(G\coprod H)/\{1_G\sim 1_H\}$ be the associated automaton.   If $L\lhd G$, let $\equiv_L$ be the equivalence relation on $G$ into cosets of $L$, which we extend to $\mathcal A$ by putting each element of $H\setminus \{1\}$ into a singleton class.
\begin{enumerate}
\item  The vertices of $\Sg_{\mathcal A}$ are the equality relation and $\{{\equiv_{\ker \lambda}}:\lambda\in B\}$.
\item  If $\Phi(x)\in \mathrm{Hom}(G,H)$, then there is an edge ${\equiv}\xrightarrow{\,\, x\,\,}{\equiv_{\ker\Phi(x)}}$ for any vertex $\equiv$ of $\Sg_{\mathcal A}$.
\item If $\Phi(x)\in \mathrm{Aut}(G)$, then $x$ labels a loop at the equality relation.
\item  The labeled subgraph of $\Sg_{\mathcal A}$ with vertices $\{{\equiv_{\ker\lambda}}:\lambda\in B\}$ and edges labeled by $X_1=\Phi^{-1}(\mathrm{Aut}(G))$ is the quotient of the Schreier graph of $A$ acting on the right of $B$ with respect to the generating set $X_1$ by the label preserving map $\lambda\mapsto {\equiv_{\ker \lambda}}$.
\end{enumerate}
In particular, all vertices of $\Sg_{\mathcal A}$ are essential and $\Vmin = \{{\equiv_{\ker \lambda}}:\lambda\in B\}$.
\end{Lem}
\begin{proof}
First note that if $h\in H\setminus \{1\}$, then $h(x)\neq k(x)$ for any $x\in X$, $k\in \mathcal A\setminus \{h\}$ as $H$ acts freely on $X$.  It follows that $h$ is not equivalent to any other state under $x\cdot {\equiv}$ for any equivalence relation $\equiv$ on $\mathcal A$ and any $x\in X$.  We deduce that, for any word $w\in X^\ast$, we have that $\equiv_w$ places $H\setminus \{1\}$ into singleton classes.    From now on we can focus on the equivalence relation that $\equiv_w$ induces on $G$, as we already understand what it looks like on $H\setminus \{1\}$.

Notice that $\{1\}\lhd G$ and $\equiv_{\{1\}}$ is the equality relation.  Next observe that if $L\lhd G$ and $\Phi(x)\in \mathrm{Hom}(G,H)$, then $x\cdot{\equiv_L} = {\equiv_{\ker\Phi(x)}}$.  Indeed, elements of $H\setminus \{1\}$ are in singleton classes by the above discussion and we have, for $g_1,g_2\in G$, that it is always the case that $g_1(x)=x=g_2(x)$.  Hence,  $g_1\mathrel{(x\cdot {\equiv_L})} g_2$ if and only if $\Phi(x)(g_1)=g_1|_x\equiv_L g_2|_x=\Phi(x)(g_2)$.  But since $\equiv_L$ restricts to the equality relation on $H$, this is equivalent to $\Phi(x)(g_1)=\Phi(x)(g_2)$.  It follows that $x\cdot {\equiv_L} = {\equiv_{\ker \Phi(x)}}$.

Observe that if $\Phi(x)\in \mathrm{Aut}(G)$ and $L\lhd G$, then $x\cdot {\equiv_L} = {\equiv_{\Phi(x)^{-1}(L)}}$.  Again the elements of $H\setminus \{1\}$ are in singleton classes by the first paragraph of the proof.  If $g_1,g_2\in G$, then $g_1(x)=x=g_2(x)$, and so we have that  $g_1\mathrel{(x\cdot {\equiv_L})} g_2$ if and only if $\Phi(x)(g_1)=g_1|_x\equiv_L g_2|_x=\Phi(x)(g_2)$, that is, if and only if $g_1$ and $g_2$ are in the same coset of $\Phi(x)^{-1}(L)$.  Thus, $x\cdot {\equiv_L} ={\equiv_{\Phi(x)^{-1}(L)}}$.   As the equality relation is $\equiv_{\{1\}}$, we deduce that $x$ labels a loop at the equality relation.  Moreover, if $\lambda\in B$, then $x\cdot {\equiv_{\ker\lambda}} = {\equiv_{\ker \lambda\circ \Phi(x)}}$.

 It now follows that the set of equivalence relations in (1) is invariant under the action of $X^\ast$, each $\equiv_{\ker \Phi(x)}$ with $\Phi(x)\in \mathrm{Hom}(G,H)$ is reachable from any of these equivalence relations (including the equality relation) by a single edge (labeled by $x$), and the Schreier graph of the action of $X_1^\ast$ on $\{\equiv_{\ker \lambda}:\lambda\in B\}$ is isomorphic as a labeled graph to the quotient of the Schreier graph of $A$ acting on the right of $B$ via the map sending $\lambda$ to $\equiv_{\ker \lambda}$.  Hence every vertex from $\{{\equiv_{\ker \lambda}}:\lambda\in B\}$ is reachable from some $\equiv_{\ker \Phi(x)}$ with $x\in X\setminus X_1$.  Therefore, every equivalence relation in (1) is reachable from the equality relation by the action of $X^\ast$.  Claims (1)--(4) follow.  The final statement is immediate from~(2),~(4) and the preceding discussion.
\end{proof}

We remark that the equality relation belongs to $\Vmin$ if and only if $\ker \lambda$ is trivial for some $\lambda\in B$.
The next proposition shows that most multispinal groups give rise to non-Hausdorff groupoids, and therefore multispinal groups provide an interesting class of examples to study in connection with simplicity.

\begin{Prop}\label{p:ms.haus}
The groupoid associated to the multispinal group $\mathfrak G$ is Hausdorff if and only if all maps in $B$ are injective.
\end{Prop}
\begin{proof}
If the groupoid is Hausdorff, then by Proposition~\ref{p:NHd}, we can find $n\geq 0$ so that all words $w\in X^\ast$ of length $n$ satisfy $|V_w|=1$.  Let $\Phi(x)\in \mathrm{Aut}(G)$.  Then,  by Lemma~\ref{l:gen.sunic.graph}, $x$ labels a loop at equality and hence $\equiv_{x^n}$ is equality.  By choice of $n$, we have $V_{x^n} = \{=\}$, and so $\equiv_{x^nw}$ is equality for any $w\in X^\ast$, i.e., all paths in $\Sg_{\mathcal A}$ labeled by $x^n$ end at equality.  But if $\lambda\in B$, then Lemma~\ref{l:gen.sunic.graph}(4) implies that $x^n$ labels a path from $\equiv_{\ker \lambda}$ to $\equiv_{\ker \lambda \circ \Phi(x)^n}$, and so $\lambda$ must be injective in order for the latter to be equality.  Conversely, if all elements of $B$ are injective, then $\Sg_{\mathcal A}$ has only the equality  vertex by Lemma~\ref{l:gen.sunic.graph} and hence $|V_w|=1$ for all words $w\in X^\ast$.  Thus the groupoid is Hausdorff by Proposition~\ref{p:NHd}.
\end{proof}

\begin{Thm}\label{t:sunic.groups}
Let $G$ and $H$ be finite groups with $H$ acting freely on $X$ and $\Phi\colon X\to \mathrm{Aut}(G)\cup \mathrm{Hom}(G,H)$
be  the data defining a multispinal group.  Let $A,B$ be as in \textrm{(S3)}.  Let $\mathfrak G$ be the corresponding self-similar multispinal group and $K$ a field.  If $\lambda\colon G\to H$ is a homomorphism, then $\widetilde\lambda\colon KG\to KH$ denotes the induced homomorphism.
\begin{enumerate}
  \item $\N{K}{\mathfrak G}{X}$ is simple if and only if $\bigcap_{\lambda\in B}\ker \widetilde\lambda =0$.
  \item If the characteristic of $K$ does not divide $|G|$, then $\N{K}{\mathfrak G}{X}$ is simple if and only if, for each irreducible representation $\rho$ of $G$ over the algebraic closure of $K$, there exists $\lambda\in B$ with $\ker \lambda\leq \ker \rho$.
  \item If the characteristic of $K$ divides $|\ker \lambda|$ for all $\lambda\in \Phi(X)\cap \mathrm{Hom}(G,H)$, then $\N{K}{\mathfrak G}{X}$ is not simple.
\end{enumerate}
\end{Thm}
\begin{Proof}
Since $\mathcal A=(G\coprod H)/\{1_G\sim 1_H\}$ is an automaton containing the nucleus of $\mathfrak G$ by Proposition~\ref{p:contains.nuc}, it suffices to verify the conditions in Theorem~\ref{t:contract.crit}.
By Lemma~\ref{l:gen.sunic.graph}, $E_{S,K}$ consists of the equations $E_{{\equiv_{\ker \lambda}}, K}$ with $\lambda \in B$. In particular, $E_{S,K}$ contains the equation $c_h=0$ for each $h\in H\setminus \{1\}$.  So any solution to $E_{S,K}$ belongs to $KG$.    But $a\in KG$ satisfies $E_{{\equiv_{\ker \lambda}}, K}$ if and only if $a\in \ker\widetilde\lambda$.  Thus $a$ satisfies $E_{S,K}$ if and only if $a\in\bigcap_{\lambda\in B}\ker \widetilde\lambda$. Since, by Lemma~\ref{l:gen.sunic.graph}, all the vertices of $\Sg_{\mathcal A}$ are essential, including the equality relation,  we deduce from  Theorem~\ref{t:contract.crit} that $\N{K}{\mathfrak G}{X}$ is simple if and only if the equations $c_g=0$, for all $g\in G$, are a consequence of the equations $E_{S,K}$, that is, $a=0$ is a consequence of $a\in\bigcap_{\lambda\in B}\ker \widetilde\lambda$.  This proves (1).

For (2), since the simplicity of $\N{K}{\mathfrak G}{X}$ depends only on the characteristic of $K$, we may assume without loss of generality that $K$ is algebraically closed.  Then the equivalence of the conditions in (1) and (2) is Proposition~\ref{p:separate.points}.

Note that if the characteristic $p$ of $K$ divides all $|\ker \lambda|$ with $\lambda\in \Phi(X)\cap \mathrm{Hom}(G,H)$, then $a=\sum_{g\in G}g$ belongs to $\ker \widetilde\lambda$ for all $\lambda\in B$ as $\widetilde\lambda(a) = |\ker \lambda|\sum_{t\in G/N}t=0$.  Thus (3) follows from (1).
\end{Proof}

Note that the kernels of the complex irreducible representations of a group $G$ can be read off the character table, and so this might be the easiest way to verify condition (2) for $K$ the field of complex numbers.  Remark~\ref{r:explicit} provides an explicit singular element in the case that there is some irreducible representation $\rho$ whose kernel contains no $\ker \lambda$ with $\lambda\in B$.

\begin{Ex}[Direct product construction]\label{ex:direct}
Let $H$ be any non-trivial finite group (acting regularly on the left of itself, that is, $X=H$) and put $G=H\times H$.  Fix a non-trivial element $h\in H$ and define $\Phi\colon H\to \mathrm{Aut}(G)\cup \mathrm{Hom}(G,H)$ by $\Phi(1)(a,b) = (b,a)$, $\Phi(h)(a,b)=b$ and $\Phi(h')$ is the trivial homomorphism $G\to H$ for $h'\neq 1,h$.  Then the corresponding multispinal group $\mathfrak G$, which is spinal and self-replicating, has a non-simple Nekrashevych algebra over every field.  To see this, note that $B$ consists of the two projections $(a,b)\mapsto a$ and $(a,b)\mapsto b$; these clearly separate points.  Consider $0\neq((1,1)-(h,1))((1,1)-(1,h))\in KG$. This belongs to the kernels of both homomorphisms $KG\to KH$ induced by the two projections $G\to H$, and so $\N{K}{\mathfrak G}{H}$ is not simple by Theorem~\ref{t:sunic.groups}(1).

When $H=\mathbb Z_2$, this construction produces the Grigorchuk-Erschler group~\cite{Grigdegree,Erschler}, which was shown by Nekrashevych (unpublished) to have a non-simple Nekrashevych algebra over every field.
\end{Ex}

\subsection{Gupta-Sidki groups, \textsf{GGS}-groups and multi-edge spinal groups}
Let  $m\geq 2$ be an integer.  Then a \textsf{GGS}-group is a multispinal group with $G=C_m$, a cyclic group of order $m$ generated by $t$, and $H=\mathbb Z_m$ (acting on itself, so $X=H$) with $\Phi(m-1)$ the identity automorphism of $G$ and $\Phi(k)\in \mathrm{Hom}(G,H)$ for all $0\leq k\leq m-2$.  Put $e_k = \Phi(k)(t)$ for $0\leq k\leq m-2$; then (S3) is satisfied if and only if $\gcd(e_0,\ldots, e_{m-2},m)=1$, which we assume from now on (cf.~\cite{BGSbranch}).     For example, when $m$ is an odd prime $p$, we obtain the Gupta-Sidki $p$-groups~\cite{GuptaSidki} by putting $e_0=1$, $e_1=-1$ and $e_k=0$ for $2\leq k\leq p-2$; these are finitely generated, infinite $p$-groups.  More generally, it is known that if $m=p^n$ with $p$ a prime, then  a \textsf{GGS}-group is a $p$-group if and only if, for each $0\leq k\leq n-1$, one has that $\sum_{j=1}^{p^{n-k} -1}e_{jp^k-1}\equiv 0\bmod p^{k+1}$.  See~\cite{BGSbranch}.  Notice that Gupta-Sidki $p$-groups with $p>3$ have associated groupoids that are not Hausdorff (by Proposition~\ref{p:ms.haus}), but in which all vertices of the simplicity graph are minimal (since $B$ contains an injective map); when $p=3$, the associated groupoid is Hausdorff by Proposition~\ref{p:ms.haus}.

\begin{Thm}
\label{t:GGS}
Let $\mathfrak G$ be a \textsf{GGS}-group over the alphabet $\mathbb Z_m$ with $m\geq 2$.  If $\Phi(k)$ is an isomorphism for some $0\leq k\leq m-2$, then $\N{K}{\mathfrak G}{\mathbb Z_m}$ is simple for all fields $K$.  Otherwise, $\N{K}{\mathfrak G}{\mathbb Z_m}$ is simple over no field $K$.  In particular, if $m$ is a prime power, e.g., in the case of Gupta-Sidki $p$-groups, the Nekrashevych algebra is simple over any field.
\end{Thm}
\begin{proof}
If $\Phi(k)$ is an isomorphism for some $k$,
then $\ker\widetilde{\Phi(k)}=0$ and so $\N{K}{\mathfrak G}{\mathbb Z_m}$ is simple by Theorem~\ref{t:sunic.groups}(1).
 This will occur, in particular, if $m$ is a prime power $p^n$ since $\gcd(e_0,\ldots, e_{m-2},m)=1$ is equivalent to one of the $e_i$ not being divisible by $p$ in this case.

On the other hand, suppose that no $\Phi(k)$ with $0\leq k\leq m-2$ is an isomorphism. Note that $B=\{\Phi(k): 0\leq k\leq m-2\}$.  The group $G$ has a faithful degree one irreducible complex representation $G\to \mathbb C^\times$ given by $\rho(t^k) = e^{2\pi ik/m}$.  Since each $\ker \lambda$ with $\lambda\in B$ is non-trivial, it follows that $\N{\mathbb C}{\mathfrak G}{\mathbb Z_m}$ is not simple by Theorem~\ref{t:sunic.groups}(2) and hence  $\N{K}{\mathfrak G}{\mathbb Z_m}$ is simple over no field $K$ by Theorem~\ref{t:contracting}.
\end{proof}

\textsf{GGS}-groups are generated by bounded automata and hence are amenable by~\cite{boundedaut}.  Therefore, the ample groupoids associated to \textsf{GGS}-groups are minimal, effective and amenable but if none of the $\Phi(k)$ are isomorphisms, then they have simple Steinberg algebras over no field.

In~\cite{multiedgespinal} and elsewhere (see the references therein) a generalization of Gupta-Sidki groups and \textsf{GGS}-groups (in the case $m$ is prime) is considered called multi-edge spinal groups.  Let $p$ be a prime and let $1\leq r\leq p-1$.  Let $G=\mathbb Z_p^r$ and $H=\mathbb Z_p$ (acting on the left of itself by the left regular representation), and let $\Phi(p-1)$ be the identity automorphism of $G$ and $\Phi(k)\in \mathrm{Hom}(G,H)$ for $0\leq k\leq p-2$.  Writing $v_1,\ldots, v_r$ for the standard basis for $G$, the condition (S3) is equivalent to the vectors $e_i=(\Phi(0)(v_i),\Phi(1)(v_i),\ldots, \Phi(p-2)(v_i))\in \mathbb Z_p^{p-1}$, for $i=1,\ldots, r$, being linearly independent.  Note that the vectors $e_1,\ldots, e_r$ determine $\Phi(0),\ldots, \Phi(p-2)$.  The associated multispinal group $\mathfrak G$ is a \emph{multi-edge spinal group}.  It is known for exactly which vectors $e_1,\ldots, e_r$ the multi-edge spinal group is a $p$-group  (they are always finitely generated and infinite, being self-replicating), cf.~\cite{multiedgespinal}.  Multi-edge spinal groups are generated by bounded automata and hence are amenable~\cite{boundedaut}.

\begin{Thm}
Let $p$  be a prime and $e_1,\ldots, e_r$ be linearly independent vectors over $\mathbb Z_p^{p-1}$.  Then the corresponding multi-edge spinal group $\mathfrak G$ has a simple Nekrashevych algebra over every field if $r=1$, and otherwise has a non-simple Nekrashevych algebra over every field.
\end{Thm}
\begin{proof}
If $r=1$, this follows from Theorem~\ref{t:GGS}.  Assume now that $2\leq r\leq p-1$ and let $G=\mathbb Z_p^r$.  If $\mu_p$ denotes the group of $p^{th}$-roots of unity in $\mathbb C$, then $\mu_p$ is a cyclic group of order $p$ and $\widehat{G}=\mathrm{Hom}(G,\mu_p)$.  Thus the kernels of the non-trivial irreducible representations of $G$ are the $\frac{p^r-1}{p-1}=1+p+p^2+\cdots + p^{r-1}>p-1$ subgroups of index $p$.  Since $B$ consists of the $p-1$ homomorphisms $\Phi(k)$ with $0\leq k\leq p-2$, we deduce that $\N{\mathbb C}{\mathfrak G}{\mathbb Z_p}$ is not simple by Theorem~\ref{t:sunic.groups}(2) and hence $\N{K}{\mathfrak G}{\mathbb Z_p}$ is not simple for any field $K$ by Theorem~\ref{t:contracting}.
\end{proof}

The above theorem provides more examples of  minimal, effective and amenable ample groupoids whose Steinberg algebras are not simple over any field.

\subsection{\v{S}uni\'{c} groups}
The following construction of finitely generated $p$-groups of intermediate growth, generalizing the Grigorchuk group, is due to \v{S}uni\'{c}.  It gives the \v{S}uni\'{c} groups~\cite{sunicgroups} associated to primitive polynomials, but we use a more field theoretic language.  Afterward, we will give the general construction.

\begin{Thm}\label{t:primitivepoly}
For every prime $p$, there is a finitely generated, infinite, contracting self-similar $p$-group of intermediate growth whose Nekrashevych algebra is simple over all fields except those of characteristic $p$ (over which it is not simple).
\end{Thm}
\begin{proof}
Let $p$ be a prime and view $\mathbb Z_p$ as the  $p$-element field. Let $F$ be a finite field of $q=p^n$ elements with $n\geq 2$.  Let $\alpha$ be a primitive element of $F$, that is, $F^\times =\langle \alpha\rangle$ (recall that the multiplicative group of a finite field is cyclic).  Let $\mathrm{Tr}\colon F\to \mathbb Z_p$ be the trace map, defined by $\mathrm{Tr}(\beta) = \sum_{i=0}^{q-1}\beta^{p^i}$.  It is a surjective group homomorphism.  Moreover, there is an isomorphism $\psi\colon F\to \mathrm{Hom}(F,\mathbb Z_p)$ given by $\psi(\gamma)(\beta) = \mathrm{Tr}(\gamma\beta)$ since the trace form $(\gamma,\beta)\mapsto \mathrm{Tr}(\gamma\beta)$ is non-degenerate (cf.~\cite[Theorem~2.24]{finitefields}).   We now define a multispinal group with $G=F$ (under addition), $H=\mathbb Z_p$ (acting on itself by the left regular representation) and $\Phi\colon \mathbb Z_p\to \mathrm{Aut}(F)\cup \mathrm{Hom}(F,\mathbb Z_p)$ given by $\Phi(0)(\beta) = \alpha\beta$ (an automorphism of $F$), $\Phi(p-1)(\beta) = \mathrm{Tr}(\beta)$  and $\Phi(k)(\beta) =0\in \mathbb Z_p$ for $0<k<p-1$ (homomorphisms $F\to \mathbb Z_p$).   In this case, $A=F^\times$ and $B=\psi(F^\times)$ is the set of all non-trivial homomorphisms $F\to \mathbb Z_p$.  It is proved in~\cite[Propositions~9,~10]{sunicgroups} (and the discussion thereafter) that the corresponding multispinal group $\mathfrak G$ is an infinite $p$-group of intermediate growth.

If $K$ is an algebraically closed field of characteristic different than $p$, then each irreducible representation $\rho$ of $G$ is of degree one, that is, given by a homomorphism $\rho\colon G\to K^\times$.  As any finite subgroup of the multiplicative group of a field is cyclic, it follows that $\ker \rho$ contains a subgroup of index $p$.  But every subgroup of index $p$ is the kernel of an element of $B$ and so the Nekrashevych algebra of $\mathfrak G$ is simple by Theorem~\ref{t:sunic.groups}(2).  On the other hand, since $p\mid |\ker \lambda|$ for all $\lambda\in B$, it follows that the Nekrashevych algebra is not simple in characteristic $p$ by Theorem~\ref{t:sunic.groups}(3).
\end{proof}

For example, if $p=2$ and $|F|=4$, then the associated group is the Grigorchuk group~\cite{grigroup}.

The above argument generalizes to arbitrary \v{S}uni\'{c} groups~\cite{sunicgroups}; we recall the construction.  Let $p$ be a prime and $f\in \mathbb Z_p[x]$ a polynomial of degree $n$ with non-zero constant term.  Let $G=\mathbb Z_p^n$, $H=\mathbb Z_p$ (acting on the left of itself). We regard elements of $G$ as column vectors, and so $\mathrm{Aut}(G)= \mathrm{GL}_n(\mathbb Z_p)$ and homomorphisms $G \to H$ are $1\times n$ matrices over $\mathbb Z_p$, that is, row vectors in $\mathbb Z_p^n$. Let $\Phi(0) = M_f$ be the companion matrix of $f$, $\Phi(p-1)$ be the projection to the last coordinate and $\Phi(k)$ be the trivial homomorphism $G\to H$ for $0<k<p-1$.  This data gives a \v{S}uni\'{c} group~\cite{sunicgroups}, denoted $G_{p,f}$. It is also a multispinal group: the conditions (S1), (S2) are obviously satisfied.
The action of $M_f$ on $\mathrm{Hom}(\mathbb Z_p^n,\mathbb Z_p)$ corresponds to its right action on row vectors, and $B$ is the orbit of the vector $(0, \ldots, 0, 1)$ under this action. That (S3) holds was observed in~\cite{sunicgroups}: namely, the row vector $(0,\ldots,0,1)$ is a cyclic vector for $M_f$.  In fact, it is noted in~\cite{sunicgroups} that up to a change of basis, any multispinal group constructed from $G$ and $H$ with exactly one automorphism and one non-trivial homomorphism $G\to H$ is of this form.

 Note that the right action of $M_f$ on row vectors induces a right action of $M_f$ on the set $\mathbb P(\mathbb Z_p^n)$ of lines through the origin in $\mathbb Z_p^n$ (i.e., the projective space over $\mathbb Z_p^n$).

If $f$ is the minimal polynomial of a primitive element $\alpha$ of a finite extension $F$ of $\mathbb Z_p$ (a so-called primitive polynomial), then $G_{p,f}$ is the self-similar group from Theorem~\ref{t:primitivepoly}.

\begin{Thm}
A \v{S}uni\'{c} group $G_{p,f}$ has a simple Nekrashevych algebra over all fields if $f$ has degree one.  Otherwise, it has a non-simple Nekrashevych algebra in characteristic $p$  and if the characteristic of $K$ is different than $p$, then $\N{K}{G_{p,f}}{\mathbb Z_p}$ is simple if and only if the companion matrix $M_f$ acts transitively on $\mathbb P(\mathbb Z_p^n)$ where $n$ is the degree of $f$.
\end{Thm}
\begin{Proof}
First note that if $f$ has degree one, then $\Phi(p-1)$ is injective and so Theorem~\ref{t:sunic.groups}(1) immediately implies that $\N{K}{G_{p,f}}{\mathbb Z_p}$ is simple.  So assume from now  on that the degree of $f$ is greater than one.
Without loss of generality, we may assume that $K$ is algebraically closed.
Two nonzero homomorphisms $\mathbb Z_p^n\to \mathbb Z_p$ have the same kernel if and only if they differ by a scalar multiple, that is, if they span the same line.
Thus the action of $M_f$ on these kernels given by $\ker \lambda\mapsto \ker \lambda\circ M_f$
corresponds to its right action on $\mathbb P(\mathbb Z_p^n)$.
Since $B$ is the orbit of the projection to the last coordinate, we see that the transitivity of the action of $M_f$ on $\mathbb P(\mathbb Z_p^n)$ is equivalent to every subgroup of index $p$ appearing as a kernel of an element of $B$.

Suppose that $K$ has characteristic $p$.  Then $p$ divides $|\ker \lambda|$ for each $\lambda\in B$ (as $n\geq 2$) and $\N{K}{G_{p,f}}{\mathbb Z_p}$ is not simple by Theorem~\ref{t:sunic.groups}(3).  If the characteristic of $K$ is different than $p$, then the group $\mu_p$ of $p^{th}$-roots of unity in $K^\times$ is a cyclic group of order $p$ and the irreducible representations of $\mathbb Z_p^n$ are the elements of $\mathrm{Hom}(\mathbb Z_p^n,\mu_p)$.  Thus every subgroup of index $p$ is the kernel of an irreducible representation, and so $\N{K}{G_{p,f}}{\mathbb Z_p}$ is simple if and only if $M_f$ acts transitively on $\mathbb P(\mathbb Z_p^n)$ by Theorem~\ref{t:sunic.groups}(2).
\end{Proof}

Note that every \v{S}uni\'{c} group is amenable, being generated by a bounded automaton~\cite{boundedaut}.
The condition that $M_f$ acts transitively on $\mathbb P(\mathbb Z_p^n)$ implies that $\mathbb Z_p^n=\bigcup_{\lambda\in B}\ker \lambda$, and hence, so long as the degree of $f$ is at least $2$, it follows from the results of~\cite[Propositions~9,~10]{sunicgroups} that whenever the Nekrashevych algebra $\N{\mathbb C}{G_{p,f}}{\mathbb Z_p}$ is simple, $G$ is an infinite $p$-group of intermediate growth.  When $f$ has degree $2$, then the index $p$-subgroups are the same as the non-trivial cyclic subgroups of $G$ and so $G_{f,p}$ will be a $p$-group if and only if  $\N{K}{G_{p,f}}{\mathbb Z_p}$ is simple outside of characteristic $p$ by~\cite[Proposition~9]{sunicgroups}.

For example, if $f=x^2+1$ and $p=2$, then the companion matrix is the permutation matrix for a transposition, which  does not act transitively on the three points of $\mathbb P(\mathbb Z_2^2)$, and hence the corresponding group $G_{2,x^2+1}$, known as the Grigorchuk-Erschler group, does not have a simple Nekrashevych algebra over any field.  This was already discussed in Example~\ref{ex:direct}.

 The group $G_{2,x^3+1}$ is known as the Grigorchuk overgroup.  The corresponding companion matrix is the permutation matrix associated to the $3$-cycle.  It does not act transitively on the $7$ points of $\mathbb P(\mathbb Z_2^3)$ and hence this group has a non-simple Nekrashevych algebra over every field.  The same situation will occur for $G_{p,x^n-1}$ for any prime $p$ and $n\geq 2$. The companion matrix is the permutation matrix for an $n$-cycle, which does not act transitively on the $\frac{p^n-1}{p-1}$ points of $\mathbb P(\mathbb Z_p^n)$, and so the Nekrashevych algebra is non-simple over every field.

The groups $G_{p,x-1}$ have a simple Nekrashevych algebra over every field since $x-1$ has degree one.  The case $p=3$ is known as the Fabrykowski-Gupta group~\cite{FabGupta}.  Other primes were considered by Grigorchuk~\cite{justinfinitebranch}.

\subsection{Contracting groups with simple Nekrashevych algebras outside of a prescribed set of primes}
Let $n>2$ be an integer and put $G=\mathbb Z_n^2$ and $H=\mathbb Z_n$ (acting on itself).   Define
\begin{gather*}
\Phi(0) = \begin{pmatrix} 1 & 1\\ 0 & 1\end{pmatrix},\quad \Phi(1) = \begin{pmatrix} 1 & 0\\ 1& 1\end{pmatrix},\\ \Phi(n-1) = \begin{pmatrix} 0 & 1\end{pmatrix} \quad  \text{and}\ \Phi(k)=\begin{pmatrix} 0 & 0\end{pmatrix}\ \text{for}\ 1<k<n-1
\end{gather*}
where we identify homomorphisms from $G$ to $H$ with row vectors in the usual way.
  Let $\mathfrak G_n$ be the corresponding multispinal group.  It is self-replicating and hence infinite.

It is well known that $\Phi(0)$ and $\Phi(1)$ generate $\mathrm{SL}_2(\mathbb Z_n)$. Indeed, $\mathbb Z_n$ is a semilocal ring and the special linear group over any semilocal ring is generated by elementary matrices~\cite[Theorem~4.3.9]{classicalgpktheory}.  Since each elementary $2\times 2$ matrix over $\mathbb Z_n$ is a power of one of the two matrices $\Phi(0)$ and $\Phi(1)$, the claim follows.

\begin{Prop}\label{p:znstuff}
The following hold:
\begin{enumerate}
\item $(x, y)$ gives a surjective homomorphism $\mathbb Z_n^2\to \mathbb Z_n$ if and only if $\gcd(x,y,n)=1$.
\item $\mathrm{SL}_2(\mathbb Z_n)$ acts transitively on the right of the set of surjective homomorphisms $\mathbb Z_n^2\to \mathbb Z_n$.
\item Any homomorphism $\lambda\colon \mathbb Z_n^2\to \mathbb Z_n$ factors through a surjective homomorphism.
\end{enumerate}
\end{Prop}
\begin{Proof}
First note that $(x,y)$ is surjective if and only if $1=xa+yb\bmod n$ for some $a,b$, that is, if and only if $1=xa+yb+cn$, which is equivalent to $\gcd(x,y,n)=1$.  If $\gcd(x,y,n) = 1$, then we can write $1=xa+yb+cn$ and so the matrix \[\begin{pmatrix} b & -a\\ x & y \end{pmatrix}\] belongs to $\mathrm{SL}_2(\mathbb Z_n)$ and takes $(0,1)$ to $(x, y)$ via right multiplication, establishing (2).  For the final item, let $(x, y)$ be any homomorphism and let $d=\gcd(x,y,n)$.  Then $(x/d, y/d)$ is a surjective homomorphism by the first item and $(x, y)= d(x/d, y/d)$.
\end{Proof}

It follows that in our previous notation, we have that $A=\mathrm{SL}_2(\mathbb Z_n)$ and $B$ consists of all surjective homomorphisms $\mathbb Z_n^2\to \mathbb Z_n$.

\begin{Thm}
Let $K$ be a field.  Then  $\N{K}{\mathfrak G_n}{\mathbb Z_n}$ is simple if and only if $p$ does not divide $n$.  Hence, if $\mathcal P$  any finite set of primes, there is a contracting, self-replicating, self-similar group whose Nekrashevych algebra is simple over precisely those fields whose characteristic does not belong to $\mathcal P$.
\end{Thm}
\begin{Proof}
Without loss of generality we may assume that $K$ is algebraically closed.
If $K$ has characteristic $p$ dividing $n$, then since each $\lambda\in B$ has kernel of size $n$, we obtain that $\N{K}{\mathfrak G_n}{\mathbb Z_n}$ is not simple over $K$ by Theorem~\ref{t:sunic.groups}(3).

If the characteristic of $K$ does not divide $n$, let $\mu_n$ be the group of $n^{th}$-roots of unity in $K$.  Then $\mu_n\cong \mathbb Z_n$ and $\widehat{G}=\mathrm{Hom}(G,\mu_n)$.  Since every homomorphism from $G$ to $\mu_n$ factors through a surjective one by Proposition~\ref{p:znstuff} and $B$ contains every surjective homomorphism to $\mathbb Z_n$, we deduce that $\N{K}{G}{\mathbb Z_n}$ is simple by Theorem~\ref{t:sunic.groups}(2).

The final statement follows because if $\mathcal P=\{2\}$, then we can take the \v{S}uni\'{c} group $G_{2,f}$ associated to a primitive polynomial (e.g., the Grigorchuk group) by Theorem~\ref{t:primitivepoly}.  Else, let $n$ be the product of the primes in $\mathcal P$ and $\mathfrak G_n$ will do the trick.
\end{Proof}

\def\malce{\mathbin{\hbox{$\bigcirc$\rlap{\kern-7.75pt\raise0,50pt\hbox{${\tt
  m}$}}}}}\def\cprime{$'$} \def\cprime{$'$} \def\cprime{$'$} \def\cprime{$'$}
  \def\cprime{$'$} \def\cprime{$'$} \def\cprime{$'$} \def\cprime{$'$}
  \def\cprime{$'$} \def\cprime{$'$}

%\bibliographystyle{abbrv}
%\bibliography{standard2}

\begin{thebibliography}{10}

\bibitem{LeavittBook}
G.~{Abrams}, P.~{Ara}, and M.~{Siles Molina}.
\newblock {\em Leavitt path algebras}.
\newblock Number 2191 in Lecture Notes in Mathematics. London: Springer, 2017.

\bibitem{multiedgespinal}
T.~Alexoudas, B.~Klopsch, and A.~Thillaisundaram.
\newblock Maximal subgroups of multi-edge spinal groups.
\newblock {\em Groups Geom. Dyn.}, 10(2):619--648, 2016.

\bibitem{Baird}
G.~R. Baird.
\newblock Congruence-free inverse semigroups with zero.
\newblock {\em J. Austral. Math. Soc.}, 20(no.1):110--114, 1975.

\bibitem{BGSbranch}
L.~Bartholdi, R.~I. Grigorchuk, and Z.~\v{S}uni\'{k}.
\newblock Branch groups.
\newblock In {\em Handbook of algebra, {V}ol. 3}, volume~3 of {\em Handb.
  Algebr.}, pages 989--1112. Elsevier/North-Holland, Amsterdam, 2003.

\bibitem{boundedaut}
L.~Bartholdi, V.~A. Kaimanovich, and V.~V. Nekrashevych.
\newblock On amenability of automata groups.
\newblock {\em Duke Math. J.}, 154(3):575--598, 2010.

\bibitem{rabbit}
L.~Bartholdi and V.~Nekrashevych.
\newblock Thurston equivalence of topological polynomials.
\newblock {\em Acta Math.}, 197(1):1--51, 2006.

\bibitem{BV05}
L.~Bartholdi and B.~Vir\'{a}g.
\newblock Amenability via random walks.
\newblock {\em Duke Math. J.}, 130(1):39--56, 2005.

\bibitem{spinalgroups}
L.~Bartholdi and Z.~\v{S}uni\'{k}.
\newblock On the word and period growth of some groups of tree automorphisms.
\newblock {\em Comm. Algebra}, 29(11):4923--4964, 2001.

\bibitem{operatorsimple1}
J.~Brown, L.~O. Clark, C.~Farthing, and A.~Sims.
\newblock Simplicity of algebras associated to \'etale groupoids.
\newblock {\em Semigroup Forum}, 88(2):433--452, 2014.

\bibitem{operatorguys2}
L.~O. Clark and C.~Edie-Michell.
\newblock Uniqueness theorems for {S}teinberg algebras.
\newblock {\em Algebr. Represent. Theory}, 18(4):907--916, 2015.

\bibitem{gridealSteinberg}
L.~O. Clark, R.~Exel, and E.~Pardo.
\newblock A generalized uniqueness theorem and the graded ideal structure of
  {S}teinberg algebras.
\newblock {\em Forum Math.}, 30(3):533--552, 2018.

\bibitem{nonhausdorffsimple}
L.~O. Clark, R.~Exel, E.~Pardo, A.~Sims, and C.~Starling.
\newblock Simplicity of algebras associated to non-{H}ausdorff groupoids.
\newblock {\em Trans. Amer. Math. Soc.}, 372(5):3669--3712, 2019.

\bibitem{cuntz}
J.~Cuntz.
\newblock Simple {$C\sp*$}-algebras generated by isometries.
\newblock {\em Comm. Math. Phys.}, 57(2):173--185, 1977.

\bibitem{Erschler}
A.~Erschler.
\newblock Boundary behavior for groups of subexponential growth.
\newblock {\em Ann. of Math. (2)}, 160(3):1183--1210, 2004.

\bibitem{Exel}
R.~Exel.
\newblock Inverse semigroups and combinatorial {$C\sp \ast$}-algebras.
\newblock {\em Bull. Braz. Math. Soc. (N.S.)}, 39(2):191--313, 2008.

\bibitem{ExPadKatsura}
R.~Exel and E.~Pardo.
\newblock Self-similar graphs, a unified treatment of {K}atsura and
  {N}ekrashevych {$\rm C^*$}-algebras.
\newblock {\em Adv. Math.}, 306:1046--1129, 2017.

\bibitem{ExelSteinbergHull}
R.~{Exel} and B.~{Steinberg}.
\newblock {Representations of the inverse hull of a 0-left cancellative
  semigroup}.
\newblock {\em arXiv e-prints}, Feb. 2018.

\bibitem{FabGupta}
J.~Fabrykowski and N.~Gupta.
\newblock On groups with sub-exponential growth functions.
\newblock {\em J. Indian Math. Soc. (N.S.)}, 49(3-4):249--256 (1987), 1985.

\bibitem{Hanoitowers}
R.~Grigorchuk and Z.~\v{S}uni\'{c}.
\newblock Schreier spectrum of the {H}anoi {T}owers group on three pegs.
\newblock In {\em Analysis on graphs and its applications}, volume~77 of {\em
  Proc. Sympos. Pure Math.}, pages 183--198. Amer. Math. Soc., Providence, RI,
  2008.

\bibitem{Grigdegree}
R.~I. Grigorchuk.
\newblock Degrees of growth of finitely generated groups and the theory of
  invariant means.
\newblock {\em Izv. Akad. Nauk SSSR Ser. Mat.}, 48(5):939--985, 1984.

\bibitem{justinfinitebranch}
R.~I. Grigorchuk.
\newblock Just infinite branch groups.
\newblock In {\em New horizons in pro-{$p$} groups}, volume 184 of {\em Progr.
  Math.}, pages 121--179. Birkh\"{a}user Boston, Boston, MA, 2000.

\bibitem{GNS}
R.~I. Grigorchuk, V.~V. Nekrashevich, and V.~I. Sushchanski{\u\i}.
\newblock Automata, dynamical systems, and groups.
\newblock {\em Tr. Mat. Inst. Steklova}, 231(Din. Sist., Avtom. i Beskon.
  Gruppy):134--214, 2000.

\bibitem{GrigZuk}
R.~I. Grigorchuk and A.~{\.Z}uk.
\newblock The lamplighter group as a group generated by a 2-state automaton,
  and its spectrum.
\newblock {\em Geom. Dedicata}, 87(1-3):209--244, 2001.

\bibitem{BasilicaGZ}
R.~I. Grigorchuk and A.~\.{Z}uk.
\newblock On a torsion-free weakly branch group defined by a three state
  automaton.
\newblock volume~12, pages 223--246. 2002.
\newblock International Conference on Geometric and Combinatorial Methods in
  Group Theory and Semigroup Theory (Lincoln, NE, 2000).

\bibitem{grigroup}
R.~I. Grigor\v{c}uk.
\newblock On {B}urnside's problem on periodic groups.
\newblock {\em Funktsional. Anal. i Prilozhen.}, 14(1):53--54, 1980.

\bibitem{GuptaSidki}
N.~Gupta and S.~Sidki.
\newblock On the {B}urnside problem for periodic groups.
\newblock {\em Math. Z.}, 182(3):385--388, 1983.

\bibitem{classicalgpktheory}
A.~J. Hahn and O.~T. O'Meara.
\newblock {\em The classical groups and {$K$}-theory}, volume 291 of {\em
  Grundlehren der Mathematischen Wissenschaften [Fundamental Principles of
  Mathematical Sciences]}.
\newblock Springer-Verlag, Berlin, 1989.
\newblock With a foreword by J. Dieudonn\'{e}.

\bibitem{algebraicExelPardo}
R.~{Hazrat}, D.~{Pask}, A.~{Sierakowski}, and A.~{Sims}.
\newblock An algebraic analogue of {Exel-Pardo} ${C}^*$-algebras.
\newblock {\em Algebr. Represent. Theory}, 2020 (to appear).
\newblock https://doi.org/10.1007/s10468-020-09973-x.

\bibitem{LamBook}
T.~Y. Lam.
\newblock {\em A first course in noncommutative rings}, volume 131 of {\em
  Graduate Texts in Mathematics}.
\newblock Springer-Verlag, New York, 1991.

\bibitem{Lawson}
M.~V. Lawson.
\newblock {\em Inverse semigroups}.
\newblock World Scientific Publishing Co. Inc., River Edge, NJ, 1998.
\newblock The theory of partial symmetries.

\bibitem{LawsonCorrespond}
M.~V. Lawson.
\newblock A correspondence between a class of monoids and self-similar group
  actions. {I}.
\newblock {\em Semigroup Forum}, 76(3):489--517, 2008.

\bibitem{Leavitt}
W.~G. Leavitt.
\newblock The module type of a ring.
\newblock {\em Trans. Amer. Math. Soc.}, 103:113--130, 1962.

\bibitem{finitefields}
R.~Lidl and H.~Niederreiter.
\newblock {\em Finite fields}, volume~20 of {\em Encyclopedia of Mathematics
  and its Applications}.
\newblock Cambridge University Press, Cambridge, second edition, 1997.
\newblock With a foreword by P. M. Cohn.

\bibitem{MarcusandLind}
D.~Lind and B.~Marcus.
\newblock {\em An introduction to symbolic dynamics and coding}.
\newblock Cambridge University Press, Cambridge, 1995.

\bibitem{selfsimilar}
V.~Nekrashevych.
\newblock {\em Self-similar groups}, volume 117 of {\em Mathematical Surveys
  and Monographs}.
\newblock American Mathematical Society, Providence, RI, 2005.

\bibitem{Nekcstar}
V.~Nekrashevych.
\newblock {$C^*$}-algebras and self-similar groups.
\newblock {\em J. Reine Angew. Math.}, 630:59--123, 2009.

\bibitem{Nekrashevychgpd}
V.~Nekrashevych.
\newblock Growth of \'etale groupoids and simple algebras.
\newblock {\em Internat. J. Algebra Comput.}, 26(2):375--397, 2016.

\bibitem{Renault}
J.~Renault.
\newblock {\em A groupoid approach to {$C\sp{\ast} $}-algebras}, volume 793 of
  {\em Lecture Notes in Mathematics}.
\newblock Springer, Berlin, 1980.

\bibitem{mygroupoidalgebra}
B.~Steinberg.
\newblock A groupoid approach to discrete inverse semigroup algebras.
\newblock {\em Adv. Math.}, 223(2):689--727, 2010.

\bibitem{groupoidprimitive}
B.~Steinberg.
\newblock Simplicity, primitivity and semiprimitivity of \'etale groupoid
  algebras with applications to inverse semigroup algebras.
\newblock {\em J. Pure Appl. Algebra}, 220(3):1035--1054, 2016.

\bibitem{simplicity}
B.~Steinberg and N.~Szak\'{a}cs.
\newblock Simplicity of inverse semigroup and \'{e}tale groupoid algebras.
\newblock {\em Adv. Math.}, 380:107611, 55, 2021.


\bibitem{sunicgroups}
Z.~\v{S}uni\'{c}.
\newblock Hausdorff dimension in a family of self-similar groups.
\newblock {\em Geom. Dedicata}, 124:213--236, 2007.

\bibitem{multispinalCstar} K.~Yoshida.
\newblock On the simplicity of $C^\ast$-algebras associated to multispinal groups.
\newblock arXiv:2102.02199, February 2021.

\end{thebibliography}
\end{document}